\documentclass{article}

\usepackage[utf8]{inputenc}     
\usepackage[english]{babel}     
\usepackage[a4paper, left=1.5in, right=1.5in, top=1.5in, bottom=1.5in]{geometry}

\usepackage{color,hyperref}
\usepackage{titlesec}
\titleformat{\paragraph}
{\normalfont\normalsize\bfseries}{\theparagraph}{1em}{}
\titlespacing*{\paragraph}
{0pt}{3.25ex plus 1ex minus .2ex}{1.5ex plus .2ex}

\usepackage{tikz}
\usetikzlibrary{arrows,positioning,shapes.geometric}
\usepackage[useregional]{datetime2}             
             
\usepackage{graphicx}           
\usepackage{amsmath,amssymb}    
\usepackage{amsthm}             
\usepackage{amsmath}
\usepackage{mathtools}          
\usepackage{enumitem}           
\usepackage{listings}           
\usepackage{todonotes}          
\usepackage{hyperref}  
\usepackage{bbold}

\newtheorem{theorem}{Theorem}[section]

\newtheorem{prop}[theorem]{Proposition}
\newtheorem{lemma}[theorem]{Lemma}

\newtheorem{remark}{Remark}
\title{Almost sure upper bound \\ for random multiplicative functions}
\author{Rachid Caich }
\date{\today}

\begin{document}

\maketitle
\begin{abstract}
    Let $\varepsilon >0$. Let $f$ be a Steinhaus or Rademacher random multiplicative function. We prove that we have almost surely, as $x \to +\infty$, 
   $$
   \sum_{n \leqslant   x} f(n) \ll \sqrt{x} (\log_2 x)^{\frac{3}{4}+ \varepsilon}.
   $$ 
\end{abstract}

\section{Introduction}
The aim of this article is to study the large fluctuations of random multiplicative functions. Random multiplicative functions has been a very attractive topic in the recent years. There are at least two models of random multiplicative functions that have been frequently studied in number theory and probability (see for example, Harper \cite{Harper}, \cite{Harper2}, \cite{Harper3}, Lau--Tenenbaum--Wu \cite{Tenenbaum}, Chatterjee--Soundararajan \cite{sound_chat}, Benatar--Nishry--Rodgers \cite{benatar}). Let $\mathcal{P}$ be the set of the prime numbers, \textit{a Steinhaus random multiplicative function} is obtained by letting $(f(p))_{p\in \mathcal{P}}$ be a sequence of independent Steinhaus random variables (i.e distributed uniformly on the unit circle $\{ |z|=1\}$), and then setting
$$ f(n):= \prod_{p^{a}||n} f(p)^{a}  \text{ for all } n\in \mathbb{N},$$
where $ p^{a}||n$ means that $p^{a}$ is the highest power of $p$ that divides $n$. \textit{A Rademacher multiplicative function} is obtained by letting $(f(p))_{p\in \mathcal{P}}$ be independent Rademacher random variables (i.e taking values $\pm 1$ with probability 1/2 each), and setting
$$
f(n)=\left\{
  \begin{array}{@{}ll@{}}
    \prod_{p |n} f(n) , & \text{if}\ n \text{ is squarefree} \\
    0, & \text{otherwise.}
  \end{array}\right.
$$
The Rademacher model was introduced by Wintner \cite{Wintner}, in 1944 as a heuristic model of M\"{o}bius function $\mu$ (see the introduction in \cite{Tenenbaum}). With a little change, one can obtain the probabilistic model for a real primitive Dirichlet character (see Granville--Soundararajan \cite{GranvilleSound2}). Steinhaus random multiplicative functions model the randomly chosen Dirichlet character $\chi $ or continuous characters $n \mapsto n^{it}$, see for example section 2 of Granville--Soundararajan \cite{GranvilleSound}. 

A classical result in the study of sums of independent random variables is the Law of the Iterated Logarithm, which predicts the almost sure size of the largest fluctuation of those sums. Let $(\xi_k)_{k \in \mathbb{N}}$ be an independent sequence of real random variables taking value $\pm1$ with probability 1/2 each. Khinchine's Law of the Iterated Logarithm consists in the almost sure statement
$$
\limsup_{N \to + \infty} \frac{|\sum_{k \leqslant  N}\xi_k|}{\sqrt{2N \log_2 N}}=1.
$$
Here and in the sequel $\log_k$ denotes the $k$-fold iterated logarithm. See for instance Gut \cite{Gut} Chapter 8 Theorem 1.1. Note that the largest fluctuations as $N$ varies (of the size $\sqrt{N \log_2 N}$) are somewhat larger than the random fluctuations one expects at a fixed point ( $\mathbb{E}\big[|\sum_{k \leqslant  N}\xi_k|\big]\asymp\sqrt{N}$).

Khinchine's theorem can't be applied in the case of random multiplicative functions, because their values are not independent. However, following Harper (see the end of the Introduction in \cite{Harper3}), we believe that a suitable version of the law of the iterated logarithm might hold.

For $f$ a multiplicative function, we denote $M_f(x):=\sum_{n\leqslant  x} f(n) $. Wintner \cite{Wintner} studied the case where $f$ is Rademacher random multiplicative function, and he was able to prove, for any fixed $\varepsilon >0$, we have almost surely
$$
M_f(x)= O( x^{1/2 +\varepsilon} )
$$
and 
$$
M_f(x) = \Omega (x^{1/2 -\varepsilon}). 
$$
This has been further improved by Erd\H{o}s \cite{Erdos} who proved that almost surely one has the bound $O(\sqrt{x}\log^A x)$ and one almost surely does not have the bound $ O( \sqrt{x}\log^{-B} x)$ for some nonnegative real numbers $A$ and $B$. In the 1980s, Hal\'{a}sz \cite{halasz} introduced several novel ideas which had made a further progress. By conditioning, coupled with hypercontractive inequalities, he proved in the Rademacher case that
$$
M_f(x)= O\big(\sqrt{x} \exp (A\sqrt{\log_2 x \log_3 x}) \big)
$$
for some nonnegative $A$. Recently, Lau--Tenenbaum--Wu \cite{Tenenbaum} (see also Basquin \cite{basquin}) improved the analysis of hypercontractive inequalities in Hal\'{a}sz’s argument, establishing, for Rademacher case, an almost sure upper bound $O(\sqrt{x}(\log_2 x)^{2+\varepsilon})$.

For the lower bound, Harper \cite{Harper3} proved that for any function $V(x)$ tending to infinity with $x$, there almost surely exists arbitrarily large values of $x$ for which
$$
\big|M_f(x) \big| \gg \frac{\sqrt{x} (\log_2 x)^{1/4}}{V(x)}.
$$
Moreover, Mastrostefano \cite{Mastrostefano} recently proved an upper bound for the sum restricted to integers that possess a large prime factor. We denote $P(n)$ to be the largest prime factor dividing $n$ with the convention $P(1)=1$. In \cite{Mastrostefano}, it was proved that we have almost surely, as $x \to +\infty$
$$
\sum_{\substack{n\leqslant x \\ P(n) > \sqrt{x}  }}f(n) \ll \sqrt{x}(\log_2 x)^{1/4+\varepsilon}.
$$
This bound is compared to the first moment: Harper, in \cite{Harper}, proved, when $x \to +\infty$
$$
\mathbb{E}\bigg[ \big| M_f(x)\big| \bigg] \asymp \frac{\sqrt{x}}{(\log_2 x)^{1/4}}.
$$
This discrepancy of a factor $\sqrt{\log_2 x}$ between the first moment and the almost sure behaviour is similar to the Law of the Iterated Logarithm for independent random variables. For this reason Harper conjectured that for any fixed $\varepsilon >0$, we might have almost surely, as $x \to +\infty$
$$ 
M_f(x)\ll \sqrt{x} (\log_2 x)^{1/4+\varepsilon}
$$
for both cases Steinhaus and Rademacher (see the introduction in \cite{Harper3} for more details). 

\noindent The main objective of this work is to improve upon the result of Lau--Tenenbaum--Wu to achieve a $3/4$ improvement.
\begin{theorem}\label{theoreme_principal_chapter_1}
Let $\varepsilon >0$. Let $f$ be a Steinhaus or Rademacher random multiplicative function. We have almost surely,  as $x \to + \infty$
\begin{equation}
    M_f(x) \ll \sqrt{x} (\log_2 x)^{\frac{3}{4}+ \varepsilon}.
\end{equation}
\end{theorem}
\section{Sketch of the proof}
Let $f$ be a Steinhaus or Rademacher multiplicative function. As is typical when aiming to establish almost sure bounds for a sum of random variables, we will focus our analysis on a sequence of "test points" (say $x_i$). These test points are chosen to be sparse, yet sufficiently close to allow for manageable control over the increments of $f$ between consecutive elements. We will then consolidate the information obtained from each test point using the first Borel–Cantelli lemma. The second step is to split  $M_f(x_i)$ according to the largest prime factor $P(n)$ of the summation variable $n$. Recall that $v_p(n)$ denotes the $p$-adic valuation of an integer $n$ (i.e. the exponent of $p$ in the product of prime factors decomposition of $n$).  Let $(y_j)_{0\leqslant  j \leqslant  J}$ a nondecreasing sequence such that $J\asymp \log_2 x_i$. The choice of $y_j$ is such that $\log y_{j-1} \sim \log y_j$. One could chose $\log y_{j} \sim \ell^{c}\log y_j$ for some constant $c$, but it will not change the result. The reason behind this choice is to approximate the factor $\frac{1}{\log x_i/z}$ that will arise latter for $x_i/y_j<z\leqslant x_i/y_{j-1}$ with $\frac{1}{\log y_j}$. 
$$M_f(x_i)=\sum_{\substack{n \leqslant  x_i \\ P(n) \leqslant  y_0 }}f(n)+ M_f^{(1)}(x_i)+M_f^{(2)}(x_i)$$
where
\begin{equation}\label{equationtoexplain}
    M_f^{(1)}(x_i):=\sum_{j=1}^J \sum_{y_{j-1} < p \leqslant  y_j} f(p) \sum_{\substack{ n\leqslant  x_i/p\\ P(n) < p}} f(n)
\end{equation}
and
\begin{equation}\label{equationtoexplain2}
    M_f^{(2)}(x_i):= \sum_{j=1}^J \sum_{\substack{y_{j-1}^2 < d \leqslant  x_i \\ v_{P(d)}(d) \geqslant 2}}^{y_{j-1},y_j} f(d) \sum_{\substack{ n\leqslant  x_i/d\\ P(n) \leqslant y_{j-1}}} f(n).
\end{equation}
where the symbol $ \sum^{y,z}$ indicates a sum restricted to integers all of whose prime factors belong to the interval $ ]y, z]$. Note that $ M_f^{(2)}(x_i)$ is equal to $0$ for Rademacher case. By choosing $y_0$ small enough, the sum $\sum_{\substack{n \leqslant  x_i \\ P(n) \leqslant  y_0 }}f(n)$ is ``small'', thus it can be neglected. We show first that $ M_f^{(2)}(x_i)$ is ``small '', this, can be done in few steps. Moreover, the expectation, conditioning on the value $f(q)$ for prime numbers $q <p$, of the quantity $f(p) \sum_{\substack{ n\leqslant  x_i/p\\ P(n) < p}} f(n)$ is equal to $0$. Therefore, $M_f^{(1)}(x_i)$ is sum of martingale differences with some kind of conditional quantity
\begin{equation}\label{varianceexplan}
    V(x_i)= \sum_{y_{0} < p \leqslant  y_J} \bigg| \sum_{\substack{ n\leqslant  x_i/p\\ P(n) < p}} f(n)\bigg|^2.
\end{equation}
By smoothing the above quantity, we get
\begin{equation}\label{expression_introduc}
    \widetilde{V}(y_j,x_i) \approx \int_{\frac{x_i}{y_{j-1}}}^{\frac{x_i}{y_j}} \frac{1}{\log (x_i/z)} \bigg| \sum_{\substack{ n\leqslant  z\\ P(n) \leqslant \frac{x_i}{z}}} f(n)\bigg|^2 \frac{{\rm d}z}{z^2}.
\end{equation}
One could compare the right-hand expression above to equation (7.1) in \cite{Mastrostefano}. In our case, the friable number $P(n)\leqslant x_i/z$ varies with $z$, making it impossible to transform into an Euler product. To eliminate this dependence, we proceed by bounding \eqref{expression_introduc} by
$$
\frac{1}{\log (y_j)}\int_{0}^{+\infty} \sup_{y_{j-1}<p\leqslant y_j} \bigg| \sum_{\substack{ n\leqslant  z\\ P(n) \leqslant p}} f(n)\bigg|^2 \frac{{\rm d}z}{z^2}.
$$
By conditioning on $ \mathcal{F}_{j-1}:=\{ f(p); p\leqslant y_{j-1}  \}$, one can bound the above quantity by  
$$
\mathbb{E}\bigg[\frac{1}{\log y_j} \int_{-\infty}^{+\infty}\frac{\big|F_{j}(1/2+it)\big|^2}{|1/2+it|^2} \,\bigg|\,  \mathcal{F}_{j-1}\bigg].
$$
This will give rise to a submartingale sequence.
A new version of Hoeffding's inequality (see Lemma~\ref{hoeffding}) allows us to reduce the problem to study the expression \eqref{varianceexplan}. The use of this inequality is exactly what gives a strong exponential upper bound. More specifically, the goal becomes to prove that we have almost surely as $x_i$ tends to infinity $V(x_i) \leqslant  x_i(\log_2 x_i)^{1/2}$. The key point here is to notice that among the terms in the sum of $V(x_i)$ that contribute significantly are the $\sum_{y_{j-1} < p \leqslant  y_j} \bigg| \sum_{\substack{ n\leqslant  x_i/p\\ P(n) < p}} f(n)\bigg|^2 $ such that $y_j$ is very ``close'' to $x_i$ (see Section \ref{Section_L3}). In fact the number of $y_j$ close to $x_i$, that contribute much in the sum, is less than $ (\log_2 x_i)^{\varepsilon/50} $. 
Thus we are reduced again to study 
\begin{equation}\label{varianceexplain2}
         \widetilde{V}(y_j,x_i):= \sum_{y_{j-1} < p \leqslant  y_j} \bigg| \sum_{\substack{ n\leqslant  x_i/p\\ P(n) < p}} f(n)\bigg|^2.
\end{equation}
By smoothing and adjusting a little bit $\widetilde{V}(y_j,x_i)$, this will give rise to a supermartingale sequence $(U_{j,i})_{j\geqslant 0} $ where $U_{0,i}=I_0$ doesn't depend on $x_i$ for each $x_i$. At the end, we use Harper's low moment result. We get roughly speaking, almost surely  $V(x_i) \leqslant  x_i(\log_2 x_i)^{1/2} $ as $x_i$ tends to infinity.
\section{Preliminary results}
\subsection{Notation}
Let's start by some definitions. Let $(\Omega, \mathcal{F}, \mathbb{P})$ be a probabilistic space. We call a \textit{filtration} every sequence $(\mathcal{F}_n)_{n \in \mathbb{N}}$ of increasing sub-$\sigma$-algebras of $\mathcal{F}$. We say that a sequence of real random variables $(Z_n)_{n \in \mathbb{N}}$ is \textit{submartingale} (resp. \textit{supermartingale}) sequence with respect to the filtration $(\mathcal{F}_n)_{n \in \mathbb{N}}$, if the following properties are satisfied:\\
- $Z_n$ is $\mathcal{F}_n$ measurable\\
- $\mathbb{E}[|Z_n|]< +\infty$\\
- $\mathbb{E}[Z_{n+1} \, | \, \mathcal{F}_n] \geqslant Z_n$ (resp. $\mathbb{E}[Z_{n+1}\, | \,\mathcal{F}_n] \leqslant  Z_n$ ) almost surely.\\
We say that $(Z_n)_{n \in \mathbb{N}}$ is martingale difference sequence with respect to the same filtration $(\mathcal{F}_n)_{n \in \mathbb{N}}$ if \\
- $Z_n$ is $\mathcal{F}_n$ measurable\\
- $\mathbb{E}[|Z_n|]< +\infty$\\
- $\mathbb{E}[Z_{n+1} \, | \,\mathcal{F}_n] =0$ almost surely.\\
An event $ E \in \mathcal{F}$ happens \textit{almost surely} if  $\mathbb{P}[E]=1$. \\
Let $Z$ be a random variable and let $\mathcal{H}_1\subset \mathcal{H}_2 \subset \mathcal{F}$ some sub-$\sigma$-algebras,  we have
$$
\mathbb{E}\big[ \mathbb{E}\big[ Z \, \big| \,\mathcal{H}_2 \big] \,\big| \, \mathcal{H}_1  \big]=\mathbb{E}\big[ Z\, \big| \, \mathcal{H}_1 \big].
$$
\subsection{Known Tools}
\begin{lemma}\label{hypercontractivity}
 Let $f$ be a Steinhaus or Rademacher random multiplicative function. For every sequence $(a_n)_{n \geqslant 1}$ of complex numbers and every positive integer $m \geqslant 1$, we have 
 $$ \mathbb{E}\left[\left|\sum_{n \geqslant 1} a_{n} f(n)\right|^{2 m}\right] \leqslant  \left(\sum_{n \geqslant 1}\left|a_{n}\right|^{2}  \tau_{2 m-1}(n)\right)^{m}, $$
where $\tau_m(n)$ is the $m$-fold divisor function.
\end{lemma}
\begin{proof}[Proof]
See Bonami \cite{Benami} Chapter III. Refer also to lemma 2 of Hal\'{a}sz \cite{halasz} for another proof. Already used by Lau--Tenenbaum--Wu in \cite{Tenenbaum}, Harper \cite{Harper2} and recently by Mastrostefano \cite{Mastrostefano}.
\end{proof}
\begin{lemma}\label{hypercontractivityresult}
Let $m \geqslant 2$ an integer. Then, uniformly for $x\geqslant 3$, we have, 
$$ \sum_{n \leqslant   x} \tau_m(n) \leqslant   x (2\log x)^{m-1}.$$
\end{lemma}
\begin{proof}[Proof]
See lemma 3.1 in \cite{benatar}.
\end{proof}
\begin{lemma}{(Parseval's identity).}\label{parseval} Let $(a_n)_{n \in \mathbb{N}^{*}}$ be sequence of complex numbers $A(s):= \sum_{n=1}^{+ \infty} \frac{a_n}{n^s}$ denotes the corresponding Dirichlet series and let also $\sigma_c$ denote its abscissa of convergence. Then for any $\sigma > \max(0,\sigma_c)$, we have 
$$ \int_{0}^{+ \infty} \frac{\big|\sum_{n \leqslant   x} a_n \big|^2}{x^{1+2\sigma}}dx=\frac{1}{2\pi}\int_{- \infty}^{+ \infty} \bigg| \frac{A(\sigma + it)}{\sigma +it} \bigg|^2 {\rm d}t.$$
\begin{proof}[Proof]
See Eq (5.26) in \cite{MontgVaugh}.
\end{proof}
\end{lemma}
\noindent We define the following parameter
\begin{equation}\label{definition_a_f}
    a_f=
     \begin{cases}
        1 & \text{ if } f \text{ is a Rademacher multiplicative function} \\
        -1  & \text{ if } f \text{ is a Steinhaus multiplicative function}.
     \end{cases}
\end{equation}
\begin{lemma}{(Euler product result)}\label{lemmarachid2}
Let $f$ be a Rademacher or Steinhaus random multiplicative function.  Let $t$ a real number and $2 \leqslant   x \leqslant   y$, we have
$$ \mathbb{E} \Bigg[ \prod_{x < p \leqslant   y} \bigg| 1+a_f\frac{f(p)}{p^{1/2+it}} \bigg|^{2a_f} \Bigg] = \prod_{x < p \leqslant   y} \bigg( 1+\frac{a_f}{p} \bigg)^{a_f}. $$
\end{lemma}
\begin{proof}[Proof]
See Mastrostefano \cite{Mastrostefano} lemma 2.4.
\end{proof}
\noindent Let $(A_n)_{n \geqslant 1}$ be sequence of events. Recall that $\limsup_{n \to + \infty} A_n = \bigcap_{n\geqslant 1} \bigcup_{k\geqslant n} A_k $.
\begin{lemma}{(Borel--Cantelli's First Lemma)}. \label{borelcantelli_chapter_1} Let $(A_n)_{n \geqslant 1}$ be sequence of events. Assuming that $\sum_{n=1}^{+ \infty}\mathbb{P}[A_n] < + \infty $ then $\mathbb{P}[\limsup_{n \to + \infty} A_n]=0$.
\end{lemma}
\begin{proof}[Proof]
See theorem 18.1 in \cite{Gut}.
\end{proof}
\begin{lemma}\label{technical_lemma_inequality}
Let $b_0,b_1,...,b_n$ be any complex numbers. We have
\begin{equation}\label{010}
    \int_{0}^1 \big|b_0+\sum_{k=1}^n {\rm e}^{i2 \pi k \vartheta}b_k \big| {\rm d}\vartheta \geqslant |b_0|.
\end{equation}
\end{lemma}
\begin{proof}[Proof]
Directly, we have
$$
\int_{0}^1 \big|b_0+\sum_{k=1}^n {\rm e}^{i2 \pi k \vartheta}b_k \big| {\rm d}\vartheta \geqslant \bigg|\int_{0}^1\bigg( b_0+\sum_{k=1}^n {\rm e}^{i2 \pi k \vartheta}b_k \bigg) {\rm d}\vartheta\bigg| = |b_0|.
$$
\end{proof}
\begin{lemma}\label{easy_supermart_Lemma}
    Suppose that the real sequence of random variables and   $\sigma$-algebras $ \{(Z_n, \mathcal{F}_n) \}_{n\geqslant 0}$ 
    is a supermartingale. Then, for every $n,m $ integers such that $n\geqslant m$, we have 
    $$
    \mathbb{E}[Z_n \, | \, \mathcal{F}_m] \leqslant Z_m.
    $$
\end{lemma}
\begin{proof}[Proof]
    See theorem 10.2.1 in \cite{Gut}.
\end{proof}
\begin{lemma}{(Doob's inequality).}\label{doob} Let $a>0$. Suppose that the real sequence of random variables and $\sigma$-algebras $ \{(Z_n, \mathcal{F}_n) \}_{n\geqslant 0}$ is a nonnegative submartingale (resp. supermartingale). Then
$$\mathbb{P}\bigg[\max_{0 \leqslant   j \leqslant   n}Z_{j} >a\bigg] \leqslant   \frac{\mathbb{E}\big[Z_n\big]}{a} \, \bigg( \text{resp. } \mathbb{P}\bigg[\max_{0 \leqslant   j \leqslant   n}Z_{j} >a\bigg] \leqslant   \frac{\mathbb{E}[Z_0]}{a}   \bigg) .$$
\end{lemma}
\begin{proof}[Proof]
See theorem 9.1 in \cite{Gut}.
\end{proof}
\begin{lemma}{(Doob's $L^r$-inequality).}\label{doob2}
Let $r>1$. Suppose that the sequence of real random variables and $\sigma$-algebras $\{(Z_n,\mathcal{F}_n) \}_{ 0\leqslant   n \leqslant   N}$ is nonnegative submartingale bounded in $L^r$. Then 
$$ \mathbb{E}\bigg[ \big(\max_{0 \leqslant   n \leqslant   N } Z_n\big)^r\bigg]\leqslant   \bigg( \frac{r}{r-1}\bigg)^r \max_{0 \leqslant   n \leqslant   N}\mathbb{E}\big[ Z_n^r\big].$$
\end{lemma}
\begin{proof}[Proof]
See theorem 9.4 in \cite{Gut}.
\end{proof}
\begin{lemma}\label{gaussian_approximation}
Let $a\geqslant 1$. For any integer $n \geqslant 1$, let $G_1,...,G_n$ be an independent real Gaussian random variables, each having mean 0 and variance between $\frac{1}{20}$ and $20$. Then
$$
\mathbb{P}\bigg[  \sum_{m=1}^k G_m \leqslant  a+2 \log k + O(1) \text{ for all } k \leqslant  n  \bigg] \asymp \min\bigg\{1,\frac{a}{\sqrt{n}}\bigg\}.
$$
\end{lemma}
\begin{proof}[Proof]
This is Probability Result 1 in \cite{Harper}.
\end{proof}
\subsection{New Tools.}
\noindent We have the following version of Azuma/Hoeffding's inequality. 
\begin{lemma}\label{hoeffding}
Let $Z=(Z_n)_{1 \leqslant   n \leqslant   N}$ be a complex martingale difference sequence with respect to a filtration $(\mathcal{F}_n)_{1 \leqslant  n \leqslant  N }$.
We assume that for each $n$, $Z_n$ is bounded almost surely.
Furthermore, assume that we have $|Z_n| \leqslant   S_n$ almost surely, where $(S_n)_{1 \leqslant  n \leqslant  N }$ is a real predictable process with respect to the same filtration (i.e for each $n$, $S_n$ is $\mathcal{F}_{n-1}$-measurable). Assume that $\sum_{1\leqslant   n \leqslant   N}S_n^2 \leqslant   T$ almost surely where $T$ is a  deterministic constant.
Then, for any $\varepsilon >0$, we have
$$ \mathbb{P}\bigg[\bigg|\sum_{1\leqslant   n \leqslant   N}Z_n\bigg|\geqslant \varepsilon\bigg] \leqslant   2\exp\bigg(\frac{-\varepsilon^2}{10T}\bigg).$$
\end{lemma}
\begin{proof}[Proof]
We define the conditional expectation $\mathbb{E}_{n}[\,.\,]:= \mathbb{E}[\,.\,| \, \mathcal{F}_n]$. Following the proof of theorem 3 in \cite{Iosif}, we define,  $g_n:= \sum_{k=1}^n Z_k$ with $g_0:=0$. We set $Z_0:=0$ and $S_0:=0$. 
Let's define now the following function, for each $n \geqslant 1$, $\lambda>0$ and $t\geqslant 0$
$$ H_n(t):= \mathbb{E}_{n-1}\Big[ \cosh{ \big(\lambda |g_{n-1}+tZ_n| \big)} \Big],$$
where $\cosh{x}:= \frac{{\rm e}^x + {\rm e}^{-x}}{2}$. Note that $H_n \geqslant 0$. We have for any positive differentiable function $u$,
$$
\begin{aligned}
(\cosh u)'' & =u'^2 \cosh u+u'' \sinh u 
\\ & \leqslant (\cosh u )(u'^2 + |u''|u).
\end{aligned}
$$
Here, we used the following inequality
$$
\sinh u \leqslant u \cosh u.
$$
Thus, by taking $u$ to be $\lambda |g_{n-1}+t Z_n |$ we get $u'^2 + |u''|u$ is less than
\begin{equation}\label{010010}
\lambda^2 \Bigg( 2\bigg(\frac{\Re(Z_n)\big( \Re(g_{n-1})+t \Re(Z_n) \big)+ \Im(Z_n)\big( \Im(g_{n-1})+t \Im(Z_n) \big)  }{|g_{n-1} + t Z_n|} \bigg)^2 + |Z_n|^2 \Bigg).    
\end{equation}
By using, $\frac{|\Re(g_{n-1})+t \Re(Z_n)|}{|g_{n-1} + t Z_n|},\frac{|\Im(g_{n-1})+t \Im(Z_n)|}{|g_{n-1} + t Z_n|} \leqslant 1$,
we get then \eqref{010010} is less than
\begin{equation}
   \lambda^2 \bigg( 2\Big(|\Re(Z_n)|+ |\Im(Z_n)| \Big)^2 + |Z_n|^2 \bigg)  
\end{equation}
and since $|\Re(Z_n)|+ |\Im(Z_n)| \leqslant \sqrt{2}|Z_n|$, we get at the end
\begin{equation}\label{eq0101}
    \begin{aligned}
  H_n''(t)  & \leqslant   \mathbb{E}_{n-1}\Big[ 5\lambda^2 |Z_n|^2 \cosh{\big(\lambda |g_{n-1}+tZ_n|}\big)\Big]
  \\ & \leqslant  5 \lambda^2 S_n^2\mathbb{E}_{n-1}\Big[  \cosh{\big(\lambda |g_{n-1}+tZ_n|}\big)\Big]
  \\ &= 5\lambda^2 S_n^2 H_n(t).
\end{aligned}
\end{equation}
Since $\mathbb{E}_{n-1}\big[  Z_n \big]=0$ and $g_{n-1}$ is $\mathcal{F}_{n-1}$-measurable, we have 
$$
\begin{aligned}
H_n'(0) &= \lambda \mathbb{E}_{n-1}\bigg[ \big( \Re (Z_n) \Re (g_{n-1}) + \Im (Z_n) \Im (g_{n-1})\big) \frac{\sinh{ (\big(\lambda |g_{n-1}|\big)}}{|g_{n-1}|}  \bigg]
\\&= \lambda  \bigg( \mathbb{E}_{n-1}\big[\Re (Z_n)\big] \Re (g_{n-1}) + \mathbb{E}_{n-1}\big[\Im (Z_n)\big] \Im (g_{n-1})\bigg) \frac{\sinh{ (\big(\lambda |g_{n-1}|\big)}}{|g_{n-1}|}  
\\&=0.
\end{aligned}
$$  
Now by lemma 3 in \cite{Iosif}, we get then
$$ H_n(t)  \leqslant   H_n(0)  \exp{ \bigg(\frac{5}{2}t^2\lambda^2 S_n^2\bigg)}.$$
In particular
$$
\begin{aligned}
H_n(1)=\mathbb{E}_{n-1} \big[\cosh{ (\lambda |g_{n}|)}\big]  & \leqslant   \mathbb{E}_{n-1} \big[\cosh{ (\lambda |g_{n-1}|)}\big]  \exp{ \bigg(\frac{5}{2}\lambda^2 S_n^2 \bigg)}
\\& = \cosh{ (\lambda |g_{n-1}|)}  \exp{ \bigg(\frac{5}{2}\lambda^2 S_n^2 \bigg)}.
\end{aligned} $$
Let's define now the following sequence, for each $n\geqslant 0$
$$ G_n:= \exp \bigg(\frac{- 5\lambda^2}{2} \sum_{k=0}^n S_k^2 \bigg) \cosh{ (\lambda |g_{n}|)} . $$
Since by assumption, $S_n$ is $\mathcal{F}_{n-1}$-measurable, we get then 
$$ \begin{aligned}
  \mathbb{E}_{n-1}\big[G_n\big] & = \mathbb{E}_{n-1}\bigg[\exp \bigg(\frac{-5 \lambda^2}{2} \sum_{k=0}^n S_k^2 \bigg) \cosh{ (\lambda |g_{n}|)} \bigg]
  \\ & =\exp \bigg(\frac{- 5\lambda^2 }{2}\sum_{k=0}^n S_k^2 \bigg)  \mathbb{E}_{n-1}\big[ \cosh{( \lambda |g_{n}|)} \big]
  \\ & \leqslant   \exp\bigg(\frac{- 5\lambda^2}{2} \sum_{k=0}^n S_k^2 \bigg) \exp\bigg( \frac{5\lambda^2}{2}S_n^2 \bigg) \cosh{ (\lambda |g_{n-1}|)}
  \\ & = G_{n-1}.
\end{aligned}  $$
We deduce then that $G_n$ is supermartingale with $\mathbb{E}G_0=1$. Thus by Doob's inequality, we have.
$$ \begin{aligned}
  \mathbb{P}\bigg[|g_N| \geqslant \varepsilon \bigg] & \leqslant   \mathbb{P}\bigg[\sup_{0 \leqslant   n \leqslant   N}|g_n| \geqslant \varepsilon \bigg]
  \\ & \leqslant   \mathbb{P}\bigg[\sup_{0 \leqslant   n \leqslant   N}G_n \geqslant \exp\bigg( \frac{-5\lambda^2T}{2}\bigg) \cosh{\lambda \varepsilon} \bigg]
  \\ & \leqslant   \frac{\exp\big( \frac{5\lambda^2T}{2}\big)}{\cosh{\lambda \varepsilon}}  \mathbb{E}G_0
  \\ & \leqslant   2\exp\bigg( \frac{5\lambda^2T}{2} - \lambda \varepsilon\bigg).
\end{aligned}$$
By choosing $\lambda$ to be $\frac{\varepsilon}{5T} $, we get the result. 
\end{proof}
\begin{lemma}\label{updated_hoeffding_chapter_1}
    Let $Z=(Z_n)_{1 \leqslant   n \leqslant   N}$ be a complex martingale difference sequence with respect to a filtration $\mathcal{F}=(\mathcal{F}_n)_{1 \leqslant  n \leqslant  N }$.
    We assume that for each $n$, $Z_n$ is bounded almost surely (let's say $ |Z_n| \leqslant   b_n$ almost surely, where $b_n$ is some real number).
    Furthermore, assume that we have $|Z_n| \leqslant   S_n$ almost surely, where $(S_n)_{1 \leqslant  n \leqslant  N }$ is a real predictable process with respect to the same filtration. We set the event $\Sigma:= \bigg\{ \sum_{1\leqslant   n \leqslant   N}S_n^2 \leqslant   T\bigg\}$ where $T$ is a  deterministic constant.
    Then, for any $\varepsilon >0$,
    $$ \mathbb{P}\bigg[\bigg\{\bigg|\sum_{1\leqslant   n \leqslant   N}Z_n\bigg|\geqslant \varepsilon\bigg\}\, \bigcap \, \Sigma \bigg] \leqslant   2\exp\bigg(\frac{-\varepsilon^2}{10T}\bigg).$$
\end{lemma}
\begin{proof}[Proof]
    We define
    $$
    \widetilde{S}_n:= S_n \mathbb{1}_{\sum_{k=1}^n S_k^2 \leqslant T}
    $$
    and
    $$
    \widetilde{Z}_n := Z_n \mathbb{1}_{\sum_{k=1}^n S_k^2 \leqslant T}.
    $$
    It is clear that $(\widetilde{Z}_n)_{1 \leqslant n\leqslant N}$ is a martingale difference sequence which is almost bounded. Note that $(\widetilde{S}_n)_{1 \leqslant n\leqslant N}$ is predictable process with respect to the filtration $\mathcal{F}$ and\\ $\sum_{ 1\leqslant n \leqslant N} \widetilde{S}_n^2 \leqslant T$. Thus, all assumptions of Lemma \ref{hoeffding} are satisfied for $(\widetilde{Z}_n)_{1 \leqslant n\leqslant N}$ and $(\widetilde{S}_n)_{1 \leqslant n\leqslant N}$. 
    Note that under the condition $\Sigma$, we have
    $$
    \sum_{1\leqslant   n \leqslant   N} Z_n = \sum_{1\leqslant   n \leqslant   N}\widetilde{Z}_n.
    $$
    thus, by Lemma \ref{hoeffding}
     $$
     \begin{aligned}
         \mathbb{P}\bigg[\bigg\{ \bigg|\sum_{1\leqslant   n \leqslant   N}Z_n\bigg|\geqslant \varepsilon \bigg\} \bigcap \, \Sigma \bigg] & \leqslant \mathbb{P}\bigg[\bigg\{\bigg|\sum_{1\leqslant   n \leqslant   N}\widetilde{Z}_n\bigg|\geqslant \varepsilon \bigg\} \bigcap \Sigma \bigg] 
         \\ &  \leqslant \mathbb{P}\bigg[\bigg|\sum_{1\leqslant   n \leqslant   N}\widetilde{Z}_n\bigg|\geqslant \varepsilon \bigg]
         \leqslant   2\exp\bigg(\frac{-\varepsilon^2}{10T}\bigg) .
     \end{aligned}$$
\end{proof}

\section{Reduction of the problem}\label{section_reduction}

From now on, we indicate by $f$ the Steinhaus or Rademacher multiplicative function.
The goal of this section is to reduce the problem to something simpler to deal with.
We want to prove that the event 
$$ \mathcal{A}:=\big\{ |M_f(x)|> 4\sqrt{x}(\log_2 x )^{3/4 + \varepsilon}     \text{, for infinitely many }x\big\},$$
holds with null probability.
As in Lau--Tenenbaum--Wu \cite{Tenenbaum}, Basquin \cite{basquin}, in Mastrostefano \cite{Mastrostefano} and, in more general, in some proofs of the Law of the Iterated Logarithm (theorem 8.1 in \cite{Gut} for example), the idea is to assess the event $\mathcal{A}$ on a suitable sequence of test points. Without introducing any change, we keep the same test points as Lau--Tenenbaum--Wu in \cite{Tenenbaum}, lemma 2.3. We take $x_i:= \lfloor {\rm e}^{i^{c_{0}}} \rfloor$, where $c_0$ is a ``small constant'' in $]0,1[$ that will be chosen in the coming Lemma~\ref{Tenenbaum_lemma2}. As Mastrostefano in \cite{Mastrostefano} at the end of Section 2, we choose on the other hand $X_{\ell}:= {\rm e}^{2^{\ell ^{K}}}$, where $K:=\frac{25}{\varepsilon}$. Set 
$$ \mathcal{A}_{\ell} := \bigg\{ \sup_{X_{\ell-1} < x_{i-1} \leqslant   X_{\ell} } \sup_{x_{i-1} < x \leqslant   x_i} \frac{|M_f(x)|}{\sqrt{x}R(x)} > 4\bigg\} $$
where $R(x):= (\log_2 x )^{3/4 + \varepsilon}$. One can easily see that $\mathcal{A} \subset \cup_{\ell  \geqslant 1}\mathcal{A}_{\ell}$.
We have the following upper bound. 
Let $\overline{\mathcal{A}}_{\ell}$ be the complement of $\mathcal{A}_{\ell}$ in simple space.
\begin{lemma}\label{Tenenbaum_lemma2}
For any fixed constant $A>0$. Recall that $x_i=\lfloor{\rm e}^{i^{c_0}}\rfloor$. There exists $c_0=c_0(A)$ small enough such that, we have almost surely, as $x_i$ tends to infinity
\begin{equation}
    \max_{x_{i-1} <x \leqslant   x_i} |M_f(x)-M_f(x_{i-1})| \ll_A \frac{\sqrt{x_i}}{(\log x_i)^A} .
\end{equation}
\end{lemma}
\begin{remark}
For $A=1$ we can take $c_0 =  \frac{1}{350}$. From now on, we take $c_0 \leqslant \frac{1}{10^3}$.
\end{remark}
\begin{proof}[Proof]
See lemma 2.3 in \cite{Tenenbaum}. Lau--Tenenbaum--Wu states the result for Rademacher case, but it can be extended easily to Steinhaus case by following the same arguments of the proof.
\end{proof}
\noindent Thus it suffices to prove that $\sum_{\ell \geqslant 1}\mathbb{P}[\mathcal{B}_{\ell}] < +\infty$ where
$$\mathcal{B}_{\ell} := \bigg\{ \sup_{X_{\ell -1} < x_{i} \leqslant   X_{\ell} }  \frac{|M_f(x_i)|}{\sqrt{x_i}R(x_i)} > 3\bigg\}. $$
\section{Upper bound of \texorpdfstring{$\mathbb{P}[\mathcal{B}_{\ell}]$}.}
\subsection{Setting up the model}

In this subsection, we give the basic idea of the approach that we are going to follow. Arguing as Lau--Tenenbaum--Wu in \cite{Tenenbaum} in the proof of lemma 3.1, with a little change in the variables, let $x_i \in ]X_{\ell -1},X_{\ell}]$, we take 
$$ y_0=\exp \bigg(2^{\ell^{K}(1-K/\ell)}\bigg)=\exp\bigg\{(\log X_{\ell})^{1-K/\ell} \bigg\} \text{ and } y_j=\exp\bigg\{{\rm e}^{j/\ell}(\log X_{\ell})^{1-K/\ell} \bigg\}.$$
Let $J$ be minimal under the constraint $y_J \geqslant   X_{\ell}$ which means $J_{\ell}=J:=\lceil K \ell^{K} \log 2 \rceil \ll ~\ell^{K}$. Note that  $\ell^{K} =\frac{1}{\log 2} \log_2 X_{\ell} \asymp  \log_2 x_i\asymp \log_2 y_j $ for any $x_i \in ]X_{\ell -1},X_{\ell}]$ and~$~1\leqslant~  j\leqslant ~J$. \\
We start by splitting $M_f(x_i)$ according to the size of the largest prime factor $P(n)$ of $n$. Let $$\Psi_f(x,y):=\sum_{\substack{n \leqslant  x \\ P(n)\leqslant  y}}f(n) \, \, \text{ and } \,\, \Psi^{\prime}_f(x,y):=\sum_{\substack{n \leqslant  x \\ P(n) < y}}f(n).$$ 
We have
$$M_f(x_i)=\Psi_f(x_i,y_0)+M_f^{(1)}(x_i)+M_f^{(2)}(x_i)  $$
with 
$$M_f^{(1)}(x_i):= \sum_{ y_0 < p  \leqslant  y_J} Y_p$$
where
$$
Y_p:=f(p) \Psi^{\prime}_f(x_i/p,p)
$$
and
$$
M_f^{(2)}(x_i):=\sum_{j=1}^J \sum_{\substack{y_{j-1}^2 < d \leqslant  x_i \\ v_{P(d)}(d) \geqslant 2}}^{y_{j-1},y_j} f(d) \Psi_f(x_i/d,y_{j-1}).
$$
One can see that $\mathcal{B}_{\ell} \subset \mathcal{B}_{\ell}^{(0)} \cup \mathcal{B}_{\ell}^{(1)}\cup \mathcal{B}_{\ell}^{(2)}$ where,
$$ \mathcal{B}_{\ell}^{(0)} :=\bigcup_{X_{\ell -1} < x_{i} \leqslant   X_{\ell} } \bigg\{   |\Psi_f(x_i,y_0)|>\sqrt{x_i}R(x_i) \bigg\}, $$
$$ \mathcal{B}_{\ell}^{(1)} :=\bigcup_{X_{\ell -1} < x_{i} \leqslant   X_{\ell} } \bigg\{ \big|M_f^{(1)}(x_i)\big|>\sqrt{x_i}R(x_i) \bigg\}  $$
and
$$ \mathcal{B}_{\ell}^{(2)} :=\bigcup_{X_{\ell -1} < x_{i} \leqslant   X_{\ell} } \bigg\{   \big|M_f^{(2)}(x_i)\big|>\sqrt{x_i}R(x_i) \bigg\} . $$
\noindent\textit{Proof of Theorem \ref{theoreme_principal_chapter_1} assuming that $ \sum_{\ell \geqslant 1} \mathbb{P}[\mathcal{B}_{\ell}^{(0)}]$, $\sum_{\ell \geqslant 1} \mathbb{P}[\mathcal{B}_{\ell}^{(1)}]$ and $\sum_{\ell \geqslant 1} \mathbb{P}[\mathcal{B}_{\ell}^{(2)}]$ converge.}\\
    If $ \sum_{\ell \geqslant 1} \mathbb{P}[\mathcal{B}_{\ell}^{(0)}]$, $\sum_{\ell \geqslant 1} \mathbb{P}[\mathcal{B}_{\ell}^{(1)}]$ and $\sum_{\ell \geqslant 1} \mathbb{P}[\mathcal{B}_{\ell}^{(2)}]$ converge then $ \sum_{\ell \geqslant 1} \mathbb{P}[\mathcal{B}_{\ell}]$ converges.
    By Borel--Cantelli's First Lemma~\ref{borelcantelli_chapter_1},  we get Theorem \ref{theoreme_principal_chapter_1}.
\hfill $\square$
\smallskip

\noindent Let's start first by dealing with  $\mathcal{B}_{\ell}^{(0)}$. 

\begin{lemma}\label{lemma_pb1}
The sum $ \sum_{\ell \geqslant 1} \mathbb{P}[\mathcal{B}_{\ell}^{(0)}]$ converges.
\end{lemma}
\begin{proof}[Proof]
Note that $\Psi_1(x,y)= \# \{ n\leqslant  x: P(n) \leqslant  y \}$. We have by Markov's inequality
$$ \mathbb{P}\big[\mathcal{B}_{\ell}^{(0)}\big] \leqslant   \sum_{X_{\ell-1} < x_i \leqslant   X_{\ell}} \mathbb{P} \bigg[\big|\Psi_f(x_i,y_0)\big|>\sqrt{x_i}R(x_i)\bigg]\leqslant    \sum_{X_{\ell-1} < x_i \leqslant   X_{\ell}} \frac{\mathbb{E}\big[\big|\Psi_f(x_i,y_0)\big|^2\big]}{x_iR(x_i)^2}.$$
However we know that $y_0 \leqslant   x_i^{\frac{1}{\log_2 x_i}}$ for $\ell$ large enough. Using \cite[Theorem 7.6]{MontgVaugh}, we have then
$$\mathbb{E}\big[\Psi_f(x_i,y_0)^2\big] \leqslant \Psi_1\big(x_i,y_0\big) \leqslant  \Psi_1\big(x_i,x_i^{\frac{1}{\log_2 x_i}}\big) \ll  x_i(\log x_i)^{-c\log_3 x_i}$$ where $c$ is an absolute constant. Thus, the sum $$\sum_{\ell \geqslant 1} \mathbb{P}\big[\mathcal{B}_{\ell}^{(0)}\big] \leqslant   \sum_{\ell \geqslant 1} \sum_{X_{\ell-1} < x_i \leqslant   X_{\ell} } (\log x_i)^{-c\log_3  x_i}$$ converges.
\end{proof}

\section{Bounding $\mathbb{P}[\mathcal{B}_{\ell}^{(2)}]$.} 
\noindent The goal of this section is to give an upper bound of $\mathbb{P}[\mathcal{B}_{\ell}^{(2)}]$. 
We set 
$$
N_{ij}(f):= \sum_{\substack{y_{j-1}^2 < d \leqslant  x_i \\ v_{P(d)}(d) \geqslant 2}}^{y_{j-1},y_j} f(d) \sum_{\substack{ n\leqslant  x_i/d\\ P(n) \leqslant y_{j-1}}} f(n)
$$
with
$$
M_{f}^{(2)}(x_i)= \sum_{j=1}^{J}N_{ij}(f).
$$
Following Lau--Tenenbaum-Wu in \cite{Tenenbaum} in section 3, let $m>1$, we have
\begin{equation}\label{equation_01_ten}
\begin{aligned}
    \mathbb{E}(|N_{ij}(f)|^{2m}\, | \, \mathcal{F}_{j-1})&\leqslant \bigg(\sum_{\substack{d\leqslant x_i \\ v_{P(d)}(d) \geqslant 2}}^{y_{j-1},y_j}  \tau_{2m-1}(d) |\Psi_f(x_i/d,y_{j-1})|^2  \bigg)^m.
    \end{aligned}
\end{equation}
We set 
$$
X_{i,j}:= \sum_{\substack{d\leqslant x_i \\ v_{P(d)}(d) \geqslant 2}}^{y_{j-1},y_j}  \tau_{2m-1}(d) |\Psi_f(x_i/d,y_{j-1})|^2.
$$
Note that 
$$
\begin{aligned}
    \mathbb{E}[X_{i,j}] & = \sum_{\substack{d\leqslant x_i \\ v_{P(d)}(d) \geqslant 2}}^{y_{j-1},y_j}  \tau_{2m-1}(d) \mathbb{E}\bigg[|\Psi_f(x_i/d,y_{j-1})|^2\bigg]
    \leqslant x_i \sum_{\substack{d\leqslant x_i \\ v_{P(d)}(d) \geqslant 2}}^{y_{j-1},y_j}  \frac{\tau_{2m-1}(d)}{d}.
\end{aligned}
$$
By writing $d=rp^2$ where $p=P(d)$, we have
\begin{equation}\label{technique1}
    \begin{aligned}
\sum_{\substack{ d\leqslant x_i\\ v_{P(d)}(d) \geqslant 2 }}^{y_{j-1},y_j} \frac{\tau_{2m-1}(d)}{d}
& \leqslant \bigg(\sum_{ r\leqslant \frac{x_i}{y_{j-1}^2} }^{y_{j-1},y_j} \frac{\tau_{2m-1}(r)}{r}+1\bigg) \sum_{y_{j-1} < p\leqslant y_j} \frac{(2m-1)^2}{p^2}
\\ & \ll \frac{m^2}{y_{j-1}} \bigg(\sum_{ r\leqslant \frac{x_i}{y_{j-1}^2} }^{y_{j-1},y_j} \frac{\tau_{2m-1}(r)}{r}+1\bigg) \sum_{y_{j-1} < p\leqslant y_j} \frac{1}{p}.
\end{aligned}
\end{equation}
Note that
$$
\sum_{y_{j-1} < p\leqslant y_j} \frac{1}{p} \ll \frac{1}{\ell} \leqslant 1.
$$
Recall that $ \tau_{2m-1}(r)\leqslant (2m-1)^{\Omega(r)}$. It is clear that
$$
\begin{aligned}
\sum_{\substack{\frac{x_i}{y_{j}^2z(1+\frac{1}{X})} \leqslant r\leqslant \frac{x_i}{y_{j-1}^2z} }}^{y_{j-1},y_j} \frac{\tau_{2m-1}(r)}{r} \leqslant \sum_{r\geqslant 1}^{y_{j-1},y_j} \frac{(2m-1)^{\Omega(r)}}{r} & = \prod_{y_{j-1} < p \leqslant y_j} \bigg(1-\frac{2m-1}{p} \bigg)^{-1}
 \leqslant {\rm e}^{cm/\ell}
\end{aligned}
$$
where $c$ is an absolute constant.
Finally, we get
\begin{equation}\label{Bound_X_ij}
    \mathbb{E}[X_{i,j}] \ll \frac{x_i m^2{\rm e}^{cm/\ell}}{y_{j-1}}. 
\end{equation}
We define the following events
$$
\mathcal{X}_{\ell}=\bigg\{ \sup_{\substack{ X_{\ell-1} < x_i \leqslant X_{\ell} \\1 \leqslant j\leqslant J}} \frac{X_{i,j}}{x_i} \leqslant \frac{1}{\ell^{10K}}  \bigg\}
\text{ and }
\mathcal{X}_{\ell,i,j}=\bigg\{  \frac{X_{i,j}}{x_i} \leqslant \frac{1}{\ell^{10K}}  \bigg\}.
$$
\begin{lemma}\label{X_ijY_ijconverges} For $m \ll \ell^K$, we have
$\sum_{\ell \geqslant 1} \mathbb{P}\big[\,\overline{\mathcal{X}_{\ell}}\,\big]$ converges.
\end{lemma}
\begin{proof}[Proof]
By using the bound \eqref{Bound_X_ij}, we get
$$
\begin{aligned}
\mathbb{P}\big[\,\overline{\mathcal{X}_{\ell}}\,\big] \leqslant \sum_{j=1}^J \sum_{ X_{\ell -1}<x_i \leqslant X_{\ell}} \frac{\ell^{10K}}{x_i} \mathbb{E}\big[  X_{ij} \big] & \ll \sum_{j=1}^J \sum_{ X_{\ell -1}<x_i \leqslant X_{\ell}}  \ell^{10K} \frac{ m^2{\rm e}^{cm/\ell}}{y_{j-1}}
\ll \ell^{11K} \frac{{\rm e}^{c_1\ell^{K}}}{2^{c_2{\rm e}^{\ell^{K}}}}
\end{aligned}
$$
where $c_1,c_2>0$ are absolute constants. Thus, the sum $\sum_{\ell \geqslant 1} \mathbb{P}\big[\,\overline{\mathcal{X}_{\ell}}\,\big]$ converges.
\end{proof}
\begin{prop}\label{B_2_converge}
The sum $ \sum_{\ell \geqslant 1}\mathbb{P}\big[\mathcal{B}_{\ell}^{(2)}\big]$ converges.
\end{prop}
\begin{proof}[Proof]
By Cauchy-Schwarz's Inequality, we have
$$\bigg| \sum_{1 \leqslant j \leqslant J} N_{ij}(f)\bigg|^{2m}\leqslant J^{2m-1} \sum_{1 \leqslant j \leqslant J} |N_{ij}(f)|^{2m}.$$
On the other hand, we have
\begin{equation}\label{inequalityX_ijY_ij}
\begin{aligned}
    \mathbb{P}\big[\mathcal{B}_{\ell}^{(2)}\big] &\leqslant \mathbb{P}\big[\mathcal{B}_{\ell}^{(2)}\cap \mathcal{X}_{\ell}\big]+\mathbb{P}\big[\,\overline{\mathcal{X}_{\ell}}\,\big].
\end{aligned}
\end{equation}
and
\begin{equation}
\begin{aligned}
      \mathbb{P}\big[\mathcal{B}_{\ell}^{(2)}\cap \mathcal{X}_{\ell}\big] &\leqslant \sum_{X_{\ell -1}< x_i \leqslant x_{\ell} } \sum_{j=1}^J \frac{\mathbb{E}\big[|N_{ij}(f)|^{2m} \, |\, \mathcal{X}_{\ell,i,j}\big]J^{2m-1}}{(x_i R(x_i)^2)^m}
    \\ & \ll \sum_{X_{\ell -1}< x_i \leqslant X_{\ell} } \sum_{j=1}^J \frac{1}{J} \bigg(\frac{c_5 m J^2 }{R(x_i)^2 \ell^{10K}} \bigg)^m
    \\ & \ll 2^{\ell^K/c_0} \bigg(\frac{c_5 m J^2 }{R(x_i)^2 \ell^{10K}} \bigg)^m.
\end{aligned}
\end{equation}
where $c_5$ is an absolute constant. By taking, $m=\ell^K$ and recall that $J \ll \ell^K$, we get then
$$
\begin{aligned}
      \mathbb{P}\big[\mathcal{B}_{\ell}^{(2)}\cap \mathcal{X}_{\ell}\big] \ll 2^{\ell^K/c_0} \bigg(\frac{c_6 }{ \ell^{15K/2}} \bigg)^{\ell^K}
\end{aligned}
$$
where $c_6$ is an absolute constant. Thus, the sum $\sum_{\ell \geqslant 1} \mathbb{P}\big[\mathcal{B}_{\ell}^{(2)}\cap \mathcal{X}_{\ell}\big]$ converges. By Lemma \ref{X_ijY_ijconverges} and the inequality \eqref{inequalityX_ijY_ij}, the sum $\sum_{\ell \geqslant 1}\mathbb{P}\big[\mathcal{B}_{\ell}^{(2)}\big] $ converges. This ends the proof.
\end{proof}

\section{Bounding \texorpdfstring{$ \mathbb{P}[\mathcal{B}_{\ell}^{(1)}]$}.}

The goal of this section is to give a bound of  $ \mathbb{P}[\mathcal{B}_{\ell}^{(1)}]$. From now on, we indicate by $p,q$ two prime numbers. We consider the filtration $\big\{ \mathcal{F}_p \big\}_{p \in\mathcal{P}}$, where $ \mathcal{F}_p$ denotes the $\sigma$-algebra generated by the random variables $f(q)$ for $q < p$. One can see that the expectation of the random variable $Y_p$ conditioned on $ \mathcal{F}_p$ gives $\mathbb{E}\big[Y_p\,|\,\mathcal{F}_p\big]=0$. Thus the sequence $(Y_p)_{p \in \mathcal{P}}$ is a martingale difference. 

\noindent Set
\begin{equation}
    V_{\ell}(x_i;f):=  \sum_{y_{0} <  p \leqslant   y_J} \big|\Psi_f^{\prime}(x_i/p,p)\big|^2.
\end{equation}

\subsection{Simplifying and smoothing \texorpdfstring{$V_{\ell}(x_i;f)$}.}\label{setion_smouthing}
The goal of this subsection is to simplify $V_{\ell}(x_i;f)$. 
Let $X$ be a large real, such that $\log X \asymp \ell^{K}$. Let $p$ be a prime and $p < t \leqslant  p(1+1/X)$.
Let 
$$ 
\Psi^{\prime}_f(x,z,y):= \sum_{\substack{ z< n \leqslant  x \\ P(n) <y }} f(n).
$$
Using the bound
$$ \big|\Psi_f^{\prime}(x_i/p,p)\big|^2 \leqslant   2 \big|\Psi_f^{\prime}(x_i/t,p)\big|^2 +2 \big|\Psi^{\prime}_f(x_i/p,x_i/t,p)\big|^2,   $$
we have $V_{\ell}(x_i;f) \leqslant   2L_{\ell}(x_i;f)+2W_{\ell}(x_i;f)$ with
\begin{equation}
    L_{\ell}(x_i;f):=\sum_{y_{0} <  p \leqslant   y_J}  \frac{X}{p}\int_{p}^{p(1+1/X)} \big|\Psi_f^{\prime}(x_i/t,p)\big|^2{\rm d}t
\end{equation}
and
\begin{equation}
    W_{\ell}(x_i;f):=\sum_{y_{0} <  p \leqslant   y_J}  \frac{X}{p}\int_{p}^{p(1+1/X)} \big|\Psi^{\prime}_f(x_i/p,x_i/t,p)\big|^2{\rm d}t.
\end{equation}
Let's start by cleaning up $L_{\ell}(x_i;f)$. We have
$$
\begin{aligned}
L_{\ell}(x_i;f)& = \sum_{y_{0} <p \leqslant   y_J}  \frac{X}{p}\int_{p}^{p(1+1/X)} \big|\Psi_f^{\prime}(x_i/t,p)\big|^2{\rm d}t
\\ & = \sum_{\substack{j=1 \\ x_i \geqslant y_{j-1}} }^J\sum_{y_{j-1} <p \leqslant   y_j}  \frac{X}{p}\int_{p}^{p(1+1/X)} \big|\Psi_f^{\prime}(x_i/t,p)\big|^2{\rm d}t.
\end{aligned}
$$
By swapping integral and summation, we have
$$ \begin{aligned}
   L_{\ell}(x_i;f)& = \sum_{\substack{j=1 \\ x_i \geqslant y_{j-1}} }^J \int_{y_{j-1}}^{y_j(1+1/X)} X \sum_{y_{j-1} < p \leqslant   y_j} \frac{1}{p} \mathbb{1}_{p < t < p(1+1/X)}|\Psi_f^{\prime}(x_i/t,p)|^2{\rm d}t
   \\ & \leqslant  \sum_{\substack{j=1 \\ x_i \geqslant y_{j-1}}}^J \int_{y_{j-1}}^{y_j(1+1/X)} X \sum_{\max \{ t/(1+1/X),y_{j-1} \} < p \leqslant  \min\{ t,y_j\}} \frac{1}{p} \big|\Psi_f^{\prime}(x_i/t,p)\big|^2{\rm d}t
   \\ & \leqslant  x_iL^{(1)}_{\ell}(x_i;f)+x_iL^{(2 )}_{\ell}(x_i;f),
\end{aligned}$$
where
\begin{equation}\label{L1}
     L^{(1)}_{\ell}(x_i;f):= \frac{1}{x_i}\sum_{\substack{j=1 \\ x_i \geqslant y_{j-1}}}^J \int_{y_{j-1}}^{y_j} X \sum_{t/(1+1/X) < p \leqslant   t} \frac{1}{p} \big|\Psi_f^{\prime}(x_i/t,p)\big|^2{\rm d}t
\end{equation}
and
\begin{equation}
    L^{(2 )}_{\ell}(x_i;f):= \frac{1}{x_i}\sum_{\substack{j=1 \\ x_i \geqslant y_{j-1}}}^J \int_{y_{j}}^{y_j(1+1/X)} X \sum_{\max \{ t/(1+1/X),y_{j-1} \} < p \leqslant   y_j} \frac{1}{p} \big|\Psi_f^{\prime}(x_i/t,p)\big|^2{\rm d}t.
\end{equation}
By changing the variable  $z:=x_i/t$ inside the integral and by simplifying by $x_i$, we find 
$$ L^{(1)}_{\ell}(x_i;f)= \sum_{\substack{j=1 \\ x_i \geqslant y_{j-1}}}^J \int_{x_i/y_{j}}^{x_i/y_{j-1}} X \sum_{\frac{x_i}{z(1+1/X)} < p \leqslant   \frac{x_i}{z}} \frac{1}{p} |\Psi_f^{\prime}(z,p)|^2\frac{{\rm d}z}{z^2} $$
and
$$
 L^{(2 )}_{\ell}(x_i;f)= \sum_{\substack{j=1 \\ x_i \geqslant y_{j-1}}}^J \int^{\frac{x_i}{y_{j}} }_{\frac{x_i}{y_j(1+1/X)}} X \sum_{ \max\{ \frac{x_i}{z(1+1/X)},y_{j-1}\}< p \leqslant   y_j} \frac{1}{p} |\Psi_f^{\prime}(z,p)|^2\frac{{\rm d}z}{z^2}.
$$
Let's first focus on $j$ such that $1 \leqslant  \frac{\log x_i}{\log y_{j-1}} \leqslant   \ell^{100K}$ in the $L^{(1)}_{\ell}(x_i;f)$'s sum. By the strong form of Mertens’ theorem (with error term given by the Prime Number Theorem) we have
$$
    \sum_{\frac{x_i}{z(1+1/X)} < p \leqslant x_i/z} \frac{1}{p}= \log \bigg( \frac{\log (x_i/z)}{\log \big(\frac{x_i}{z(1+1/X)}\big)}\bigg) +O \bigg({\rm e}^{-C\sqrt{\log \big(\frac{x_i/z}{1+1/X} \big)}}\bigg)
$$
where $C$ is an absolute constant. Since $x_i/y_{j} < z \leqslant  x_i/y_{j-1}$ and $\log y_{j} = {\rm e}^{\frac{1}{\ell}}  \log y_{j-1}$, we get $\log (x_i/z) \asymp \log y_{j-1}$. On the other hand, by assumption, 
$\log X \asymp \log_2 y_{j-1} \asymp \ell^{K}$, we have then
\begin{equation}\label{Mertens}
    \sum_{\frac{x_i}{z(1+1/X)} < p \leqslant   \frac{x_i}{z}} \frac{1}{p} \ll \frac{1}{X\log y_{j-1}}.
\end{equation}
Thus, we have
$$ \begin{aligned}
 &\sum_{\substack{j=1\\ 1 \leqslant  \frac{\log x_i}{\log y_{j-1}} \leqslant   \ell^{100K}}}^J \int_{x_i/y_{j}}^{x_i/y_{j-1}} X \sum_{\frac{x_i}{z(1+1/X)} < p \leqslant   \frac{x_i}{z}} \frac{1}{p} |\Psi_f^{\prime}(z,p)|^2\frac{{\rm d}z}{z^2}
 \\ & \leqslant   \sum_{\substack{j=1\\ 1 \leqslant  \frac{\log x_i}{\log y_{j-1}} \leqslant   \ell^{100K}}}^J  \int_{x_i/y_{j}}^{x_i/y_{j-1}} X \sum_{\frac{x_i}{z(1+1/X)} < p \leqslant   \frac{x_i}{z}} \frac{1}{p} \sup_{\frac{x_i}{z(1+1/X)} < q \leqslant   \frac{x_i}{z}} |\Psi_f^{\prime}(z,q)|^2\frac{{\rm d}z}{z^2}
 \\ & \ll \sum_{\substack{j=1\\ 1 \leqslant  \frac{\log x_i}{\log y_{j-1}} \leqslant   \ell^{100K}}}^J  \frac{1}{\log y_{j}} \int_{x_i/y_{j}}^{x_i/y_{j-1}} \sup_{\frac{x_i}{z(1+1/X)} < q \leqslant   \frac{x_i}{z}} |\Psi_f^{\prime}(z,q)|^2\frac{{\rm d}z}{z^2}.
\end{aligned}
$$
Since 
$$
\begin{aligned}
    |\Psi_f^{\prime}(z,q)|^2 & = \bigg| \Psi_f(z,x_i/z) - \sum_{\substack{n \leqslant z \\ \frac{x_i}{z(1+1/X)} < P(n) \leqslant x_i/z}}  f(n)  + \sum_{\substack{n \leqslant z \\ \frac{x_i}{z(1+1/X)} < P(n) < q}}  f(n)   \bigg|^2
    \\ & \leqslant 4 \big| \Psi_f(z,x_i/z)\big|^2 +4 \bigg|  \sum_{\substack{n \leqslant z \\ \frac{x_i}{z(1+1/X)} < P(n) \leqslant x_i/z}}  f(n)     \bigg|^2 + 4 \bigg|  \sum_{\substack{n \leqslant z \\ \frac{x_i}{z(1+1/X)} < P(n) < q}}  f(n)   \bigg|^2
\end{aligned}
$$
We define,
$$
M_{\ell}(x_i,y_j;f)=M_{\ell}(x_i,y_j;f) := \frac{1}{\log y_{j}} \int_{x_i/y_{j}}^{x_i/y_{j-1}}  |\Psi_f(z,x_i/z)|^2\frac{{\rm d}z}{z^2},
$$
$$
\begin{aligned}
    \lambda^{(2)}_{\ell}(x_i,y_j;f)& :=  \frac{1}{\log y_{j}} \int_{x_i/y_{j}}^{x_i/y_{j-1}} \sup_{\frac{x_i}{z(1+1/X)} \leqslant q \leqslant   \frac{x_i}{z}}  \bigg|  \sum_{\substack{n \leqslant z \\ \frac{x_i}{z(1+1/X)} < P(n) < q}}  f(n)   \bigg|^2\frac{{\rm d}z}{z^2},
    \\ \lambda^{(3)}_{\ell}(x_i,y_j;f) &:= \frac{1}{\log y_{j}} \int_{x_i/y_{j}}^{x_i/y_{j-1}}  \bigg|  \sum_{\substack{n \leqslant z \\ \frac{x_i}{z(1+1/X)} < P(n) \leqslant \frac{x_i}{z}}}  f(n)     \bigg|^2\frac{{\rm d}z}{z^2},
    \\      L^{(12)}_{\ell}(x_i;f)& :=  \sum_{\substack{j=1\\ \frac{\log x_i}{\log y_{j-1}} > \ell^{100K}}}^J \int_{x_i/y_{j}}^{x_i/y_{j-1}} X \sum_{\frac{x_i}{z(1+1/X)} < p \leqslant   \frac{x_i}{z}} \frac{1}{p} |\Psi_f^{\prime}(z,p)|^2\frac{{\rm d}z}{z^2}.
\end{aligned}
$$
Since the number of $y_j$ such that $ 1 \leqslant  \frac{\log x_i}{\log y_{j-1}} \leqslant   \ell^{100K} $ is less than $100K\ell \log \ell$ then
$$
\begin{aligned}
\frac{L_{\ell}(x_i;f)}{x_i} \ll & \ell \log \ell \sup_{\substack{ y_j \\ 1 \leqslant  \frac{\log x_i}{\log y_{j-1}} \leqslant   \ell^{100K}}} M_{\ell}(x_i,y_j;f)+ \ell\log \ell \Bigg( \sum_{k=2}^3\sup_{\substack{ y_j \\ 1 \leqslant  \frac{\log x_i}{\log y_{j-1}} \leqslant   \ell^{100K}}} \lambda^{(k)}_{\ell}(x_i,y_j;f) \Bigg) 
\\ & +  L^{(12)}_{\ell}(x_i;f) + L^{(2)}_{\ell}(x_i;f)
\\ \leqslant & \ell \log \ell \sup_{\substack{ y_j}} M_{\ell}(x_i,y_j;f)+ \ell\log \ell \Bigg( \sum_{k=2}^3\sup_{\substack{ y_j }} \lambda^{(k)}_{\ell}(x_i,y_j;f) \Bigg) 
\\ & +  L^{(12)}_{\ell}(x_i;f) + L^{(2)}_{\ell}(x_i;f).
\end{aligned}
$$ 
Thus
\begin{equation}\label{inequality9901}
    \begin{aligned}
    \frac{V_{\ell}(x_i;f)}{x_i} & \ll\ell \log \ell M_{\ell}(x_i,y_j;f)+  \ell \log \ell \Bigg( \sum_{k=2}^3\sup_{\substack{ 0\leqslant  j\leqslant J }} \lambda^{(k)}_{\ell}(x_i,y_j;f)\Bigg) \\ & \,\,\,\,\,\,\,\, +  L^{(12)}_{\ell}(x_i;f) + L^{(2)}_{\ell}(x_i;f) + \frac{ W_{\ell}(x_i;f)}{x_i}.
\end{aligned}
\end{equation}
It turns out that $ M_{\ell}(x_i,y_j;f)$ makes the most contribution in above right hand sum. The others terms of the sum will be bounded straightforwardly. 
Let $T(\ell)=\ell^{10}$ a positive real parameter depending on $\ell$. We define, for each $k \in \{2,3 \}$ the following probabilities 
\begin{equation}\label{lambda(k)}
    \mathbb{P}_{\ell}^{\lambda,k}:= \mathbb{P}\bigg[  \sup_{X_{\ell -1} < x_i \leqslant   X_{\ell}}  \sup_{\substack{ 0 \leqslant  j\leqslant J }} \lambda_{\ell}^{(k)}(x_i,y_j,f)> \frac{T(\ell)}{ \ell^{K/2} \ell \log \ell} \bigg],
\end{equation}
\begin{equation}\label{lambda(1)}
    \mathbb{P}_{\ell}^{\lambda,1}:= \mathbb{P}\bigg[  \sup_{X_{\ell -1} < x_i \leqslant   X_{\ell}}  \sup_{\substack{ 0 \leqslant  j\leqslant J }} M_{\ell}(x_i,y_j,f)> \frac{T(\ell) \ell^{K/2} }{  \ell \log \ell} \bigg],
\end{equation}
and we define, as well
$$\mathbb{P}_{\ell}^{(12)}:= \mathbb{P}\bigg[ \sup_{X_{\ell -1} < x_i \leqslant   X_{\ell} } L^{(12)}_{\ell}(x_i;f)> \frac{T(\ell)}{ \ell^{K/2}} \bigg] ,$$
$$\mathbb{P}_{\ell}^{(2)}:= \mathbb{P}\bigg[ \sup_{X_{\ell -1} < x_i \leqslant   X_{\ell} } L^{(2)}_{\ell}(x_i;f)> \frac{T(\ell)}{ \ell^{K/2}} \bigg] $$
and
$$
    \mathbb{P}_{\ell}^{W}:=\mathbb{P}\bigg[\sup_{X_{\ell -1} < x_i \leqslant   X_{\ell}} \frac{\ell^{K/2}  W_{\ell}(x_i;f)}{x_i} >1  \bigg].
$$
\subsection{Bounding \texorpdfstring{$\mathbb{P}_{\ell}^{W}$}.}
The goal of this subsection is to prove the convergence of the sum $\sum_{\ell \geqslant 1}\mathbb{P}_{\ell}^{W}$.
\begin{prop}\label{prop2}
We have, $\sum_{\ell \geqslant 1}  \mathbb{P}_{\ell}^{W}$ converges. 
\end{prop}
\noindent To prove the above proposition, we need the following Lemma.
\begin{lemma}\label{lemmachoseX}
Let $r >1$ be an integer. We have 
\begin{equation}\label{lemmasmallinterval}
    \mathbb{P}_{\ell}^{W} \ll_r  \sum_{X_{\ell -1} < x_i \leqslant   X_{\ell}}  \bigg(\frac{\ell^{K/2}}{\log x_i} \bigg)^r.
\end{equation}
\end{lemma}
\begin{proof}[Proof]

Using Markov's inequality for the power $r >1$, we have 
$$\begin{aligned}
  \mathbb{P}_{\ell}^{W} & \leqslant \frac{1}{T(\ell)^r}  \sum_{X_{\ell -1} < x_i \leqslant   X_{\ell}} \bigg(\frac{\ell^{K/2}}{x_i}\bigg)^r \mathbb{E}\big[W_{\ell}(x_i;f)^r\big]
  \\ & \leqslant \sum_{X_{\ell -1} < x_i \leqslant   X_{\ell}} \bigg(\frac{\ell^{K/2}}{x_i}\bigg)^r \mathbb{E}\big[W_{\ell}(x_i;f)^r\big].
\end{aligned} $$
Following the steps of Harper's work \cite{Harper2} and more recently Mastrostefano  \cite{Mastrostefano} in section 6, we begin first by applying Minkowski's inequality, we then bound the above expectation with
$$\begin{aligned}
  \mathbb{E}\big[W_{\ell}(x_i;f)^r\big]& \leqslant   \bigg(  \sum_{\substack{y_{0}<p \leqslant   y_J \\  p\leqslant   x_i}}  \bigg( \mathbb{E}\bigg[\bigg(\frac{X}{p}\int_{p}^{p(1+1/X)} \big|\Psi^{\prime}_f(x_i/p,x_i/t,p)\big|^2  {\rm d}t \bigg)^{r}\bigg] \bigg)^{\frac{1}{r}}\bigg)^r.
\end{aligned} $$
Then by applying H\"{o}lder's inequality on the normalized integral $ \frac{X}{p} \int^{p(1+1/X)}_p {\rm d}t$ with parameters $1/r$ and $(r-1)/r$ we can bound the above sum with
\begin{equation}
    \leqslant  \bigg(  \sum_{\substack{y_{0}<p \leqslant   y_J \\  p\leqslant   x_i} } \bigg( \frac{X}{p}\int_{p}^{p(1+1/X)} \mathbb{E}\bigg[\big|\Psi^{\prime}_f(x_i/p,x_i/t,p)\big|^{2r}\bigg]  {\rm d}t  \bigg)^{\frac{1}{r}}\bigg)^r.
\end{equation}
Now let's focus on bounding the $2r$-th moment of partial sum of $f$ over short interval. By arguing as Harper \cite{Harper2} in the proof of proposition 2, we observe that when $\frac{x_i}{p}-\frac{x_i}{p(1+1/X)}<1 $, the interval $]x_i/p(1+1/X),x_i/p] $ contains at most one integer. Hence
$$ \frac{X}{p}\int_{p}^{p(1+1/X)} \mathbb{E}\bigg[\big|\Psi^{\prime}_f(x_i/p,x_i/t,p)\big|^{2r} \bigg] {\rm d}t  \leqslant   1.$$
Otherwise we have $p \leqslant   \frac{x_i}{1+X}$, and by applying Cauchy Schwarz's inequality, we get 
$$ 
\begin{aligned}
  \mathbb{E}\bigg[\big|\Psi^{\prime}_f(x_i/p,x_i/t,p) \big|^{2r}\bigg]  &  \leqslant   \sqrt{   \mathbb{E}\bigg[\big|\Psi^{\prime}_f(x_i/p,x_i/t,p)\big|^{2}  \bigg] \mathbb{E}\bigg[\big|\Psi^{\prime}_f(x_i/p,x_i/t,p)\big| ^{2(2r-1)}  \bigg]   }.
\end{aligned}$$
Since $t\leqslant   p(1+1/X)$ then $\frac{x_i}{p}-\frac{x_i}{t}<  \frac{x_i}{p(1+X)}$. We get then
$$\mathbb{E}\bigg[\big|\Psi^{\prime}_f(x_i/p,x_i/t,p)\big|^{2} \bigg] \leqslant   \sum_{\substack{\frac{x_i}{t}<  n \leqslant   \frac{x_i}{p}\\ P(n) <p}} 1 \ll \frac{x_i}{pX}. $$
For the second expectation, we apply Lemma~\ref{hypercontractivity} 
$$ \mathbb{E}\bigg[\big|\Psi^{\prime}_f(x_i/p,x_i/t,p)\big|^{2(2r-1)} \bigg]\ll \bigg(\sum_{\substack{\frac{x_i}{t}<  n \leqslant   \frac{x_i}{p}\\ P(n) <p}}\tau_{4r-3}(n) \bigg)^{2r-1}.  $$ 
By applying Lemma~\ref{hypercontractivityresult}, we deduce
$$ \ll \bigg( \frac{x_i}{p}(\log x_i)^{4r-4} \bigg)^{2r-1}.$$
We get at the end in the case where $p \leqslant   \frac{x_i}{1+X}$
\begin{equation}
    \mathbb{E}\bigg[\big|\Psi^{\prime}_f(x_i/p,x_i/t,p)\big|^{2r}\bigg] \ll_q \bigg(\frac{x_i}{p}\bigg)^r \frac{(\log x_i)^{(2r-2)(2r-1)}}{\sqrt{X}}. 
\end{equation}
By choosing $X=(\log x_i)^{8r^2-8r+4}$, we conclude that 
\begin{equation}\label{equation2}
\begin{aligned}
  \mathbb{E}\big[W_{\ell}(x_i;f)^{r}\big] & \ll_r   \bigg( \sum_{\frac{x_i}{1+X} < p \leqslant   x_i }  1 + x_i \frac{(\log x_i)^{\frac{(2r-2)(2r-1)}{r}}}{X^{\frac{1}{2r}}}\sum_{  p\leqslant   \frac{x_i}{1+X} } \frac{1}{p} \bigg)^r
  \\ & \ll_r \bigg(\frac{x_i}{\log x_i} \bigg)^r.
   \end{aligned}
\end{equation}
Then
$$\mathbb{P}_{\ell}^{W} \ll_r  \sum_{X_{\ell -1} < x_i \leqslant   X_{\ell}}  \bigg(\frac{\ell^{K/2}}{\log x_i} \bigg)^r,$$
this ends the proof.
\end{proof}
\noindent \textit{ Proof of Proposition \ref{prop2}.}

By choosing $r > 1/c_0$ where $c_0$ is the constant chosen in Lemma~\ref{Tenenbaum_lemma2}, we get the convergence of $\sum_{\ell \geqslant 1} \mathbb{P}_{\ell}^{W}.$

\hfill $\square$

\subsection{Bounding \texorpdfstring{$\mathbb{P}_{\ell}^{\lambda,1}$}. }\label{low_moment_estimates}
The goal of this subsection is to prove the convergence of the sum of  $\sum_{\ell \geqslant 1}\mathbb{P}_{\ell}^{\lambda,1}$. 
We have for all $1\leqslant j\leqslant J$
$$
M_{\ell}(x_i,y_j;f)\leqslant U_j:= \frac{1}{\log y_j} \int_{0}^{+\infty} \max_{y_{j-1}<p\leqslant y_j} |\Psi_f(z,p)|^2\frac{{\rm d}z}{z^2}
$$
and
$$
M_{\ell}(x_i,y_0;f) \leqslant U_0:= \frac{1}{\log y_0} \int_{0}^{+\infty}  |\Psi_f(z,y_0)|^2\frac{{\rm d}z}{z^2}.
$$
We set
\begin{equation}\label{I_j}
    I_{j}:= \bigg(\frac{\log y_{j}}{\log y_0}\bigg)^{-1/\ell^K} \frac{1}{\log y_{j}}\int_{- \infty}^{+\infty}  \bigg|\frac{F_{j}(1/2+it)}{1/2+it}\bigg|^2 {\rm d}t
\end{equation}
where $F_{j}(s)=\prod_{p \leqslant y_j}\big(1+ a_f\frac{f(p)}{p^s} \big)^{a_f}$.
The reason why we added the factor $\big(\frac{\log y_{j}}{\log y_0}\big)^{-1/\ell^K}$ is to make sure that $(I_{j})_{0 \leqslant  j \leqslant  J}$ is supermartingale sequence with respect to the filtration $(\mathcal{F}_{y_j})_{0 \leqslant  j \leqslant  J}$.\\
Recall $T(\ell)=\ell^{10}$, we denote 
\begin{equation}\label{definition_T_1}
    T_1(\ell)=\frac{T(\ell)}{\ell \log \ell}.
\end{equation}
Define $\mathcal{S}$ to be the event $\big\{ I_{j} \leqslant   \frac{T_1(\ell)^{1/2}}{\ell ^{K/2}} \text{ for all } 0 \leqslant j\leqslant J \big\}$ and $\mathcal{S}_j:=\big\{ I_{j} \leqslant   \frac{T_1(\ell)^{1/2}}{\ell ^{K/2}}  \big\} $.
Now we have
$$
\begin{aligned}
    \mathbb{P}_{\ell}^{\lambda,1} & \leqslant  \mathbb{P}\bigg[ \bigcup_{0 \leqslant j\leqslant J}\bigg\{ 
    U_j \geqslant  \frac{T(\ell) \ell^{K/2} }{  \ell \log \ell}   \bigg\} \bigcap \big\{ \mathcal{S}\big\}\bigg]+ \mathbb{P}\big[\,\overline{\mathcal{S}}\,\big]
    \\ & \leqslant \sum_{j=0}^{J} \mathbb{P}\bigg[ \bigg\{ 
    U_j \geqslant  \frac{T(\ell) \ell^{K/2} }{  \ell \log \ell}   \bigg\} \bigcap \big\{ \mathcal{S}_{j-1}\big\}\bigg] + \mathbb{P}\big[\,\overline{\mathcal{S}}\,\big].
\end{aligned}
$$
\noindent Let start by treating
$$
\widetilde{\mathbb{P}}_j:=\mathbb{P}\bigg[ \bigg\{ 
U_j \geqslant  \frac{T(\ell) \ell^{K/2} }{  \ell \log \ell}   \bigg\} \bigcap \big\{ \mathcal{S}_{j-1}\big\}\bigg].
$$
By Markov's inequality, we have
$$
\begin{aligned}
\widetilde{\mathbb{P}}_j & \leqslant \mathbb{P}\bigg[ \bigg\{ 
U_j \geqslant  \frac{T(\ell) \ell^{K/2} }{  \ell \log \ell}   \bigg\}\,  \bigg| \, \big\{ \mathcal{S}_{j-1}\big\}\bigg] 
\\ & \leqslant \frac{\ell \log \ell}{T(\ell) \ell^{K/2}} \mathbb{E}\big[ U_j \, \big|\,\mathcal{S}_{j-1}\big].
\end{aligned}
$$
Before going further, we need the following lemmas
\begin{lemma}\label{submartingale_q_0}
    Let $z\geqslant 1$, we consider
$$
X_q(z,q_0):=   \sum_{\substack{n \leqslant z \\ q_0 < P(n) \leqslant q}}  f(n) .
$$
We have $(|X_q(z,q_0)|)_{q \in \mathcal{P}}$ is a submartingale under the filtration $(\mathcal{F}_q)$.
\end{lemma}
\begin{proof}[Proof]
Indeed, let $q<p$ be two consecutive prime numbers. For Rademacher case, we have 
$$
\begin{aligned}
\mathbb{E}\big[ \big|X_p(z,q_0) \big| \,\big|\, \mathcal{F}_p\big] & =\mathbb{E}\bigg[ \big|X_q(z,q_0)+f(p)X_q(z/p,q_0) \big| \,\bigg|\, \mathcal{F}_p\bigg]
\\ & = \frac{1}{2}\big|X_q(z,q_0)+X_q(z/p,q_0) \big|+\frac{1}{2}\big|X_q(z,q_0)-X_q(z/p,q_0) \big|
\\ & \geqslant \big|X_q(z,q_0) \big|.
\end{aligned}
$$ 
For Steinhaus case, let $n$ to be the smallest integer such that $\frac{z}{p^n} <1$. By applying Lemma \ref{technical_lemma_inequality}, we have
$$
\begin{aligned}
\mathbb{E}\big[ \big|X_p(z,q_0) \big| \,\big|\, \mathcal{F}_p\big] & =\mathbb{E}\bigg[ \big|X_q(z,q_0)+\sum_{k=1}^n f(p)^kX_q(z/p^k,q_0) \big| \,\bigg|\, \mathcal{F}_p\bigg]
\\& =\int_{0}^1 \big|X_q(z,q_0)+\sum_{k=1}^n {\rm e}^{i2 \pi k \vartheta}X_q(z/p^k,q_0) \big| {\rm d}\vartheta
\\ & \geqslant \big|X_q(z,q_0)\big|
\end{aligned}
$$
Then in both cases Rademacher and Steinhaus, we have $$\mathbb{E}\big[ \big|X_p(z,q_0) \big| \,\big|\, \mathcal{F}_p\big] \geqslant \big|X_q(z,q_0) \big|.$$
\end{proof}
\begin{lemma}\label{supermartinagale_lemma_chapter_1}
    For $\ell$ large enough, the sequence $(I_j)_{j\geqslant 0}$ is supermartingale with respect to the filtration $(\mathcal{F}_{y_j})_{j\geqslant 0}$.
\end{lemma}
\begin{proof}[Proof]
We have
    $$
    \begin{aligned}
        \mathbb{E}\big[I_{j} \, \big| \mathcal{F}_{y_{j-1}}\,\big] & = \frac{1}{\log y_j}\bigg(\frac{\log y_{j}}{\log y_0}\bigg)^{-1/\ell^K}\int_{-\infty}^{+\infty} \mathbb{E}\bigg[ \bigg|  \frac{F_j(1/2+it)}{1/2+it} \bigg|^2 \, \big| \, \mathcal{F}_{y_{j-1}} \,\, \bigg] {\rm d}t 
        \\ & = \frac{1}{\log y_j}\bigg(\frac{\log y_{j}}{\log y_0}\bigg)^{-1/\ell^K}\int_{-\infty}^{+\infty} \mathbb{E}\bigg[\prod_{y_{j-1}<p\leqslant y_j} \bigg|1+a_f\frac{f(p)}{p^{1/2+it}} \bigg|^{2a_f} \bigg] \bigg|  \frac{F_{j-1}(1/2+it)}{1/2+it} \bigg|^2   {\rm d}t.
    \end{aligned}
    $$
    We have, using Lemma \ref{lemmarachid2}
    $$
    \mathbb{E}\bigg[\prod_{y_{j-1}<p\leqslant y_j} \bigg|1+a_f\frac{f(p)}{p^{1/2+it}} \bigg|^{2a_f} \bigg]:= \prod_{y_{j-1}<p\leqslant y_j} \bigg(1+\frac{a_f}{p}\bigg)^{a_f}.
    $$
   In the sake of readability, we set
   $$
   b(a_f,j):= \prod_{y_{j-1}<p\leqslant y_j} \bigg(1+\frac{a_f}{p}\bigg)^{a_f}{\rm e}^{-1/\ell} \bigg(\frac{\log y_{j}}{\log y_{j-1}}\bigg)^{-1/\ell^K}.
   $$
    By collecting the previous computations together, we find
    $$
    \begin{aligned}
        \mathbb{E}\big[I_{j} \, \big| \mathcal{F}_{y_{j-1}}\,\big] \leqslant b(a_f,j)I_{j-1}.
    \end{aligned}
    $$
    To end the proof it suffices to prove that $b(a_f,j)\leqslant 1$. Recall that $\frac{\log y_j}{\log y_{j-1}}={\rm e}^{1/\ell}$.
    For $\ell$ large enough, we have
$$
\begin{aligned}
b(a_f,j) & = {\rm e}^{-1/\ell} \big({\rm e}^{\frac{1}{\ell}}\big)^{-1/\ell^K}\prod_{y_{j-1} <   p \leqslant y_{j}} \bigg(1 +a_f\frac{1}{p} \bigg)^{a_f} 
\\ & = \exp \bigg( -\frac{1}{\ell} + \frac{-1}{\ell^{K+1}}+a_f\sum_{y_{j-1} <   p \leqslant y_{j}} \log  \bigg(1 +a_f\frac{1}{p} \bigg)  \bigg) 
 \\ & = \exp \bigg(\frac{-1}{\ell^{K+1}} + O\bigg(\sum_{y_{j-1} \leqslant   p < y_{j}} \frac{1}{p^2}\bigg) +O\bigg( {\rm e}^{-C\sqrt{\log y_{j-1}}}\bigg) \bigg)
\end{aligned}
$$
for some constant $C$. We have $\sum_{y_{j-1} \leqslant  p < y_{j}} \frac{1}{p^2}\ll \frac{1}{y_{j-1}}\ll \frac{1}{y_0}$.
We get
$$
\begin{aligned}
 b(a_f,j) & = \exp \bigg(\frac{-1}{\ell^{K+1}} +O\bigg( {\rm e}^{-C\sqrt{\log y_{0}}}\bigg) \bigg)
 \\ & = \exp \bigg(\frac{-1}{\ell^{K+1}} +O\bigg( {\rm e}^{-2^{\frac{\ell^{K}}{4}}}\bigg) \bigg).
\end{aligned}
$$
Thus, for large $\ell$, we have $b(a_f,j)\leqslant 1$.
It follows then that $\mathbb{E}\big[I_{j} \, \big| \,\mathcal{F}_{y_{j-1}}\big] \leqslant ~ I_{j-1}.$
\end{proof}
\noindent Using Lemma \ref{submartingale_q_0} for $q_0=1$, followed by Lemma \ref{doob2} for $r=2$, we get \\$$\begin{aligned}
    \mathbb{E}\big[ \max_{y_{j-1}<p\leqslant y_j} |\Psi_{f}(z,p)|^2  \, \big|\,\mathcal{S}_{j-1}\big] & \leqslant 4 \max_{y_{j-1}<p\leqslant y_j} \mathbb{E}\big[  |\Psi_{f}(z,p)|^2  \, \big|\,\mathcal{S}_{j-1}\big] \\ & = 4 \mathbb{E}\big[  |\Psi_{f}(z,y_j)|^2  \, \big|\,\mathcal{S}_{j-1}\big]
\end{aligned}$$
Recall that $\bigg(\frac{\log y_{j}}{\log y_0}\bigg)^{-1/\ell^K} \asymp 1$.  We have then
$$
\begin{aligned}
     \mathbb{E}\big[ U_j \, \big|\,\mathcal{S}_{j-1}\big] & =  \frac{1}{\log y_j}\int_{0}^{+\infty}\mathbb{E}\big[ \max_{y_{j-1}<p\leqslant y_j} |\Psi_{f}(z,p)|^2  \, \big|\,\mathcal{S}_{j-1}\big] \frac{{\rm d}z}{z^2}
     \\& \leqslant  \frac{4}{\log y_j}\int_{0}^{+\infty}\mathbb{E}\big[  |\Psi_{f}(z,y_j)|^2  \, \big|\,\mathcal{S}_{j-1}\big] \frac{{\rm d}z}{z^2}
     \\& \ll \mathbb{E}\bigg[ \bigg(\frac{\log y_{j}}{\log y_0}\bigg)^{-1/\ell^K} \frac{1}{\log y_j}\int_{0}^{+\infty}  |\Psi_{f}(z,y_j)|^2  \,  \frac{{\rm d}z}{z^2}\bigg|\,\mathcal{S}_{j-1}\bigg]
     \\& =\mathbb{E}\big[I_j\,\big|\, \mathcal{S}_{j-1}\big].
\end{aligned}
$$
Now by using Lemma \ref{supermartinagale_lemma_chapter_1}, we have $\mathbb{E}\big[I_j\,\big|\, \mathcal{F}_{j-1}\big]\leqslant I_{j-1}.$
Thus we have 
$$
\mathbb{E}\big[ U_j \, \big|\,\mathcal{S}_{j-1}\big] \ll \mathbb{E}\big[I_j\,\big|\, \mathcal{S}_{j-1}\big] = \mathbb{E}\bigg[\mathbb{E}\big[I_j\,\big|\, \mathcal{F}_{j-1}\big]\,\big|\, \mathcal{S}_{j-1}\bigg] \leqslant \mathbb{E}\big[I_{j-1}\,\big|\, \mathcal{S}_{j-1}\big]\leqslant \frac{T_1(\ell)^{1/2}}{\ell^{K/2}}.
$$
Thus, we get at the end
\begin{equation}\label{sum_P_j_tilde}
    \sum_{j=1}^{J} \Tilde{\mathbb{P}}_j\ll \sum_{j=1}^{J} \frac{T_1(\ell)^{1/2}}{\ell^{K/2}} \frac{\ell \log \ell}{T(\ell) \ell^{K/2}} \ll \frac{T_1(\ell)^{1/2}\ell \log \ell}{T(\ell)}=\frac{\sqrt{\ell \log \ell}}{T(\ell)^{1/2}}.
\end{equation}
Since $T(\ell) = \ell^{10}$, we have $\sum_{\ell \geqslant 1} \frac{\sqrt{\ell \log \ell}}{T(\ell)^{1/2}}$ converges.\\
Let's now consider $ \mathbb{P}\big[ \mathcal{S}\big]$.\\
We define $\mathcal{A}:= \big\{ I_0\leqslant \frac{T_1^{1/4}(\ell)}{\ell^{K/2}} \big\}$.
\begin{lemma}\label{P(S)}
    For $\ell$ large enough, we have $\mathbb{P}\big[\,\overline{\mathcal{S}}\,\big]\ll\frac{1}{(T_1(\ell))^{1/6}}$.
\end{lemma}
\begin{proof}[Proof]
    Indeed, we have $$\mathbb{P}\big[\,\overline{\mathcal{S}}\,\big] \leqslant \mathbb{P}\bigg[\max_{0 \leqslant j\leqslant J} I_j > \frac{(T_1(\ell))^{1/2}}{\ell^{K/2}}\, \bigg|\, \mathcal{A} \bigg] +  \mathbb{P}\big[\,\overline{\mathcal{A}}\,\big].$$
Recall, by Lemma $\ref{supermartinagale_lemma_chapter_1}$, the sequence $(I_j)$ is supermartingale. Thus by lemma $\ref{doob}$, we get
$$
\mathbb{P}\bigg[\max_{0 \leqslant j\leqslant J} I_j > \frac{(T_1(\ell))^{1/2}}{\ell^{K/2}}\, \bigg|\, \mathcal{A} \bigg] \leqslant \frac{\ell^{K/2}}{(T_{1}(\ell))^{1/2}} \mathbb{E}\big[I_0\, \big|\, \mathcal{A}\big] \leqslant \frac{1}{(T_{1}(\ell))^{1/4}}.
$$
On the other hand we have, From Key Proposition 1 and Key Proposition 2 in \cite{Harper}, as it is done in the paragraph entitled: “Proof of the upper bound in Theorem 1, assuming Key Propositions 1 and 2”, for Steinhaus case, we have by taking $y_0=x^{1/{\rm e}}$, $q=2/3$ and $k=0$
\begin{equation}\label{Harperineq}
\mathbb{E}[I_0^{2/3}]\ll  \mathbb{E}\bigg[\bigg(\frac{1}{\log y_0}\int_{- \infty}^{+\infty}  \bigg|\frac{F_{0}(1/2+it)}{1/2+it}   \bigg|^2 {\rm d}t\bigg)^{2/3} \bigg] \ll \frac{ 1}{\ell^{K/3}} .
\end{equation}
In the case of Rademacher, the inequality \eqref{Harperineq} follows directly from Key Propositions 3 and 4 in [12], using the same proof as in the Steinhaus case and by taking the same values $y_0=x^{1/{\rm e}}$, $q=2/3$ and $k=0$. The second part of the lemma follows easily from Markov's inequality
$$\mathbb{P}\bigg[I_{0} > \frac{T_1(\ell)^{1/4}}{\ell ^{K/2}} \bigg]\leqslant   \frac{\ell ^{\frac{K}{3}}\mathbb{E}[I_{0}^{\frac{2}{3}}]}{T_1(\ell)^{\frac{1}{6}}} \ll \frac{1}{T_1(\ell)^{1/6}}. $$
\end{proof}
\begin{prop}\label{converge_P_lambda_1}
$ \sum_{\ell \geqslant 1 }\mathbb{P}^{\lambda,1}_{\ell} $ converges.
\end{prop}
\begin{proof}[Proof]
By gathering Lemma \ref{P(S)} and inequality \eqref{sum_P_j_tilde}, we get 
$$
\mathbb{P}^{\lambda,1}_{\ell} \ll \frac{1}{(T_1(\ell))^{1/6}}.
$$
Since $T_1(\ell) =\frac{T(\ell)}{\ell \log \ell} \gg \ell^8$.
Thus $ \sum_{\ell \geqslant 1 }\mathbb{P}^{\lambda,1}_{\ell} $ converges.
\end{proof}
\subsection{Bounding $\mathbb{P}_{\ell}^{\lambda,2}$ and $\mathbb{P}_{\ell}^{\lambda,3}$.}\label{Section_Lambda_2_3}
\begin{lemma}\label{Lemma_convergence_lambda_2_3}
    For $k=2,3$, the sum  $\sum_{\ell \geqslant 1} \mathbb{P}_{\ell}^{\lambda,k}$ converges.
\end{lemma}
\begin{proof}[Proof]
    We have
$$
\begin{aligned}
    \mathbb{P}_{\ell}^{\lambda,k} \leqslant  \frac{\ell^K}{T_1(\ell)^2}\sum_{X_{\ell -1} < x_i \leqslant   X_{\ell}}  \sum_{\substack{ 0 \leqslant  j\leqslant J }} \mathbb{E}\bigg[ \bigg(\lambda_{\ell}^{(k)}(x_i,y_j;f) \bigg)^2\bigg].
\end{aligned}
$$
For fixed $z$, we consider
$$
X_q(z):=   \sum_{\substack{n \leqslant z \\ \frac{x_i}{z(1+1/X)} < P(n) \leqslant q}}  f(n)   .
$$
By Lemma \ref{submartingale_q_0}, for $q_0=\frac{x_i}{z(1+1/X)}$, we have $(|X_q(z)|)_{q \in \mathcal{P}}$ is a submartingale under the filtration $(\mathcal{F}_p)$. Thus,
by Cauchy-Schwarz's inequality followed by Doob's $L^4$-inequality (Lemma \ref{doob2}), we get 
$$
\begin{aligned}
     \mathbb{E}\bigg[ \bigg(\lambda_{\ell}^{(2)}(x_i,y_j;f) \bigg)^2\bigg] & \leqslant  \frac{1}{(\log y_j)^2}  \bigg(\int_{x_i/y_{j}}^{x_i/y_{j-1}}  \frac{{\rm d}z}{z} \bigg) 
     \\ & \,\,\,\,\,\,\times \int_{x_i/y_{j}}^{x_i/y_{j-1}} \mathbb{E}\Bigg[ \Bigg(\sup_{\frac{x_i}{z(1+1/X)} \leqslant q \leqslant   \frac{x_i}{z}}  \bigg|  \sum_{\substack{n \leqslant z \\ \frac{x_i}{z(1+1/X)} < P(n) < q}}  f(n)   \bigg|^4\Bigg)\Bigg] \frac{{\rm d}z}{z^3}
    \\ & \ll \frac{1}{\log y_j} \int_{x_i/y_{j}}^{x_i/y_{j-1}} \mathbb{E}\Bigg[   \bigg|  \sum_{\substack{n \leqslant z \\ \frac{x_i}{z(1+1/X)} < P(n) \leqslant \frac{x_i}{z}}}  f(n)   \bigg|^4\Bigg] \frac{{\rm d}z}{z^3}.
\end{aligned}
$$
By Cauchy-Schwarz's inequality, we get in the case of $\lambda_{\ell}^{(3)}(x_i,y_j;f)$
$$
\begin{aligned}
    \mathbb{E}\bigg[ \bigg(\lambda_{\ell}^{(3)}(x_i,y_j;f) \bigg)^2\bigg] & \ll \frac{1}{\log y_j} \int_{x_i/y_{j}}^{x_i/y_{j-1}} \mathbb{E}\Bigg[   \bigg|  \sum_{\substack{n \leqslant z \\ \frac{x_i}{z(1+1/X)} < P(n) \leqslant \frac{x_i}{z}}}  f(n)   \bigg|^4\Bigg] \frac{{\rm d}z}{z^3}.
\end{aligned}
$$
To give an upper bound of the fourth moment, we use Lemma \ref{hypercontractivity} and Lemma \ref{hypercontractivityresult}. This gives
$$
\begin{aligned}
    \mathbb{E}\Bigg[   \bigg|  \sum_{\substack{n \leqslant z \\ \frac{x_i}{z(1+1/X)} < P(n) \leqslant \frac{x_i}{z}}}  f(n)   \bigg|^4\Bigg] & \leqslant \bigg(  \sum_{\substack{n \leqslant z \\ \frac{x_i}{z(1+1/X)} < P(n) \leqslant \frac{x_i}{z}}}  \tau_3(n)   \bigg)^2
    \\ & \leqslant \bigg( 3 \sum_{\frac{x_i}{z(1+1/X)} <p\leqslant \frac{x_i}{z}} \,\,\,\,\sum_{\substack{k \leqslant \frac{z}{p} }}  \tau_3(k)   \bigg)^2
    \\ & \ll z^2 (\log z)^4 \bigg(  \sum_{\frac{x_i}{z(1+1/X)} <p\leqslant \frac{x_i}{z}} \frac{1}{p}   \bigg)^2.
\end{aligned}
$$
Since $ x_i/y_{j} <z \leqslant x_i/y_{j-1}$, from \eqref{Mertens}, we have 
$$
\sum_{\frac{x_i}{z(1+1/X)} <p\leqslant \frac{x_i}{z}} \frac{1}{p} \ll \frac{1}{X \log y_j}.
$$
We get then, in both cases ($k=2 \text{ or } 3$)
$$
\begin{aligned}
    \mathbb{E}\bigg[ \bigg(\lambda_{\ell}^{(k)}(x_i,y_j;f) \bigg)^2\bigg] & \ll \frac{(\log x_i)^4}{X^2 (\log y_j)^3} \int_{x_i/y_{j}}^{x_i/y_{j-1}} \frac{{\rm d}z}{z}
     \ll \frac{(\log x_i)^4}{X^2 (\log y_0)^2}.
\end{aligned}
$$
Thus at the end, we get
$$
\begin{aligned}
    \mathbb{P}_{\ell}^{\lambda,k} & \ll \frac{\ell^{2K}}{T_1(\ell)^2 (\log y_0)^2} \sum_{X_{\ell -1} < x_i \leqslant   X_{\ell}} \frac{(\log x_i)^4}{X^2}.
\end{aligned}
$$
Now recall from the proof of Lemma \ref{lemmachoseX} in Section \ref{setion_smouthing} that $X = (\log x_i)^{8r^2-8r+4}$ where $r > 1/c_0>1$ ($c_0$ is defined in Lemma \ref{Tenenbaum_lemma2}). Since $2\times(8r^2-8r+4)-4> r $ for $r> 1$ , we have then $c_0(2\times(8r^2-8r+4)-4)>1$.
Thus 
$$
\begin{aligned}
    \mathbb{P}_{\ell}^{\lambda,k} & \ll \frac{\ell^{2K}}{T_1(\ell)^2 (\log y_0)^2} \sum_{ i \leqslant   2^{\ell^K/c_0}} \frac{1}{i^{c_0(2\times(8r^2-8r+4)-4)}}
    \\ & \ll \frac{\ell^{2K}}{T_1(\ell)^2 (\log y_0)^2}.
\end{aligned}
$$
Since $\log y_0 = \frac{2^{\ell^k}}{2^{K\ell^{K-1}}}$ and $T_1(\ell)\geqslant \ell^4$, then $\sum_{\ell \geqslant1}\mathbb{P}_{\ell}^{\lambda,k}$ converges which ends the proof.
\end{proof}
\subsection{Bounding \texorpdfstring{$\mathbb{P}^{(12)}_{\ell}$}.}\label{Section_L3}
The goal of this subsection is to give an upper bound bound of $\mathbb{P}^{(12)}_{\ell}$.

\begin{lemma}\label{lemmatop99}
For $\ell$ large enough and $x_i \in ]X_{\ell-1},X_{\ell}]$, we have
\begin{equation}
    \mathbb{E}\big[ L^{(12)}_{\ell}(x_i;f) \big] \ll \ell^K {\rm e}^{\ell^{-99K}/2}.
\end{equation}

\end{lemma}
\begin{proof}[Proof]
Note first that for $p  \leqslant \frac{x_i}{z} \leqslant y_j $,
$$
\mathbb{E} \big[ |\Psi^{\prime}_f(z,p)|^2 \big] \ll z {\rm e}^{\frac{ -\log z}{2\log p}} \leqslant z {\rm e}^{\frac{ -\log z}{2\log y_j}}.
$$
By using again the bound $$\sum_{\frac{x_i}{z(1+1/X)} < p \leqslant   \frac{x_i}{z}}\frac{1}{p} \ll \frac{1}{X\log y_{j-1}}$$
we get
$$
\begin{aligned}
\mathbb{E}\big[ L^{(12)}_{\ell}(x_i;f) \big] & \ll \sum_{\substack{j=1\\ \frac{\log x_i}{\log y_{j-1}} > \ell^{100K}}}^J \int_{x_i/y_{j}}^{x_i/y_{j-1}} X \sum_{\frac{x_i}{z(1+1/X)} < p \leqslant   \frac{x_i}{z}} \frac{1}{p} z {\rm e}^{\frac{ -\log z}{2\log y_j }} \frac{{\rm d}z}{z^2} 
\\ & \ll \sum_{\substack{j=1\\ \frac{\log x_i}{\log y_{j-1}} > \ell^{100K}}}^J \frac{1}{\log y_j} \int_{x_i/y_{j}}^{x_i/y_{j-1}} \frac{1}{z^{1+\frac{1}{2\log y_j} }}  {\rm d}z  \ll \ell^K {\rm e}^{-\ell^{99K}/2}.
\end{aligned}
$$
\end{proof}
\begin{lemma}\label{lemmatop31} For $T\geqslant 1$, we have $ \sum_{\ell \geqslant 1}\mathbb{P}^{(12)}_{\ell} $  converges.
\end{lemma}
\begin{proof}[Proof]
For $\ell$ large and by applying Lemma~\ref{lemmatop99}, we have
$$
\begin{aligned}
 \mathbb{P}^{(12)}_{\ell} & \leqslant \sum_{X_{\ell -1} < x_{i} \leqslant   X_{\ell}} \frac{\ell^{K/2}}{T} \mathbb{E}\big[ L^{(12)}_{\ell}(x_i;f) \big]
 \\ & \ll  \sum_{X_{\ell -1} < x_{i} \leqslant   X_{\ell}} \ell^{K/2}  \ell^K {\rm e}^{-\ell^{99K}/2}
 \\ &\ll 2^{\ell^K /c_0} \ell^{3K/2} {\rm e}^{-\ell^{99K}/2} \ll {\rm e}^{-\ell^{98K}}.
\end{aligned}
$$
Thus 
$\sum_{\ell \geqslant 1} \mathbb{P}^{(12)}_{\ell}$ converges.
\end{proof}
\subsection{Bounding \texorpdfstring{$\mathbb{P}^{(2)}_{\ell}$}.}
The goal of this subsection is to bound $\mathbb{P}^{(2)}_{\ell}$. 
 \begin{lemma}\label{lemmatop32}
For $T\geqslant 1$, the sum $\sum_{\ell \geqslant 1} \mathbb{P}^{(2)}_{\ell}$ converges.
 \end{lemma}
 \begin{proof}[Proof]
 By bounding the expectation, we get 
$$
\begin{aligned}
\mathbb{E}\big[L^{(2)}_{\ell}(x_i;f)\big] & \ll \sum_{j=1}^J \int_{\frac{x_i}{y_j(1+1/X)}}^{\frac{x_i}{y_j}} X \sum_{\max\big( \frac{x_i}{z(1+1/X)},y_{j-1}\big) < p \leqslant   y_j} \frac{1}{p} {\rm e}^{-\frac{\log z}{2 \log p}}\frac{{\rm d}z}{z}
\\ & \leqslant   \sum_{j=1}^J \int_{\frac{x_i}{y_j(1+1/X)}}^{\frac{x_i}{y_j}} X \sum_{\max\big( \frac{x_i}{z(1+1/X)},y_{j-1}\big) < p \leqslant   y_j} \frac{1}{p} {\rm e}^{-\frac{\log z}{2 \log y_j}}\frac{{\rm d}z}{z}.
\end{aligned}
$$
Note that $y_j \leqslant \frac{x_i}{z}$, then by applying again \eqref{Mertens}, we have
$$
\sum_{\max\big( \frac{x_i}{z(1+1/X)},y_{j-1}\big) < p \leqslant   y_j} \frac{1}{p} \leqslant  \sum_{ \frac{x_i}{z(1+1/X)} < p \leqslant   \frac{x_i}{z}} \frac{1}{p} \ll \frac{1}{X \log y_j} .
$$
Let's back now to the expectation, we get
$$
\begin{aligned}
\mathbb{E}\big[L^{(2)}_{\ell}(x_i;f)\big]  & \ll \sum_{j=1}^J \frac{1}{\log y_j} \int_{\frac{x_i}{y_j(1+1/X)}}^{\frac{x_i}{y_j}} \frac{1}{z^{1+\frac{1}{2\log y_j}}} {\rm d}z
\\ & \ll \sum_{j=1}^J \frac{1}{{\rm e}^{\frac{\log x_i}{2 \log y_j}}}\bigg( {\rm e}^{\frac{\log (1+1/X)}{2 \log y_j}} -1 \bigg)
\\ & \ll \sum_{j=1}^J \frac{1}{X \log y_j}.
\end{aligned}
$$
Recall that $X$ is chosen to be $(\log x_i)^{8r^2-8r+4}$ at the end of Lemma~\ref{lemmachoseX}'s proof. We denote $r^{\prime}:=8r^2-8r+4 $. Note that $r^{\prime}>1$. We deduce then 
$$
\begin{aligned}
\mathbb{P}^{(2)}_{\ell} & \ll \frac{\ell^{K/2}}{T}\sum_{X_{\ell-1} <x_i \leqslant   X_{\ell}} \sum_{j=1}^J \frac{1}{X \log y_j}
\\ & \ll \frac{\ell^{K/2}}{T} 2^{\ell^{K}/c_0} \frac{2^{K\ell^{K-1}}}{2^{r^{\prime}(\ell-1)^{K} + \ell^K}}.
\end{aligned}
$$
Since $r^{\prime}>r>1/c_0$, we get the convergence of the quantity $$\sum_{\ell \geqslant 1}   2^{\frac{K}{2} \frac{\log \ell}{\log 2}  +\frac{1}{c_0}\ell^{K} +K \ell^{K-1}- r^{\prime} (\ell-1)^{K}-\ell^K}.$$ Thus, the sum $\sum_{\ell \geqslant 1} \mathbb{P}^{(2)}_{\ell} $ converges.
\end{proof}
 
\subsection{Convergence of \texorpdfstring{$\sum_{\ell \geqslant 1} \mathbb{P}\big[ \mathcal{B}^{(1)}_\ell\big]$}.}

From \eqref{inequality9901}, there exists a constant $C$ such that
$$
\begin{aligned}
    C\frac{ V_{\ell}(x_i;f)}{\ell^{K/2}x_i} & \leqslant  \frac{\ell \log \ell}{\ell^{K/2}}\sup_{\substack{1 \leqslant y_j\leqslant J }} M_{\ell}(x_i,y_j;f)+  \log \ell \Bigg( \sum_{k=2}^3\sup_{\substack{1 \leqslant y_j\leqslant J }} \lambda^{(k)}_{\ell}(x_i,y_j;f)  \Bigg) \\ & \,\,\,\,\,\,\,\, + \frac{ W_{\ell}(x_i;f)}{x_i} +   L^{(12)}_{\ell}(x_i;f) + L^{(2)}_{\ell}(x_i;f) 
\end{aligned}
$$
We set 
$$
\mathbb{P}^{V}_{\ell}:=P\bigg[ \sup_{ X_{\ell-1}<x_i \leqslant X_{\ell}} \frac{\ell^{K/2} V_{\ell}(x_i;f)}{x_i} > \frac{6T(\ell)}{C} \bigg].$$
\begin{lemma}\label{convergence_P_v_l}
    We have $\sum_{\ell \geqslant 1} \mathbb{P}^{V}_{\ell}$ converges.
\end{lemma}
\begin{proof}[Proof]
    It suffices to observe that
    $$\mathbb{P}^{V}_{\ell}  \leqslant \mathbb{P}_{\ell}^{\lambda,1} + \mathbb{P}_{\ell}^{\lambda,2}+ \mathbb{P}_{\ell}^{\lambda,3} + \mathbb{P}_{\ell}^{(12)} + \mathbb{P}_{\ell}^{(2)} + \mathbb{P}_{\ell}^{W}.$$
Since $T_1(\ell) \geqslant \ell^4$, 
\begin{itemize}
    \item by Proposition \ref{converge_P_lambda_1}, $\sum_{\ell \geqslant 1} \mathbb{P}_{\ell}^{\lambda,1}$ converges,
    \item by Lemma \ref{Lemma_convergence_lambda_2_3}, $\sum_{\ell \geqslant 1} \mathbb{P}_{\ell}^{\lambda,2}$ and $\sum_{\ell \geqslant 1} \mathbb{P}_{\ell}^{\lambda,3}$ converge,
    \item by Lemma \ref{lemmatop31}, $\sum_{\ell \geqslant 1} \mathbb{P}_{\ell}^{(12)} $ converges,
    \item by Lemma \ref{lemmatop32}, $\sum_{\ell \geqslant 1} \mathbb{P}_{\ell}^{(2)} $ converges,
    \item by Proposition \ref{prop2}  $\sum_{\ell \geqslant 1} \mathbb{P}_{\ell}^{W}$ converges.
\end{itemize}
\end{proof}

\noindent Now we set the following event 
$$
\Sigma := \bigg\{ \forall x_i \in ] X_{\ell -1}, X_{\ell}], V_{\ell}(x_i;f) \leqslant \frac{6\ell^{K/2}T(\ell)x_i}{C} \bigg\}
$$
and for each $x_i \in ] X_{\ell -1}, X_{\ell}]$
$$
\Sigma_i := \bigg\{  V_{\ell}(x_i;f) \leqslant \frac{6\ell^{K/2}T(\ell)x_i}{C} \bigg\}.
$$
Note that $\mathbb{P}\big[\, \overline{\Sigma}\,\big]=\mathbb{P}_{\ell}^V$.
\begin{prop}
    The sum $\sum_{\ell \geqslant 1} \mathbb{P}\big[ \mathcal{B}^{(1)}_\ell\big]$ converges.
\end{prop}
  \begin{proof}[Proof]
      Recall that $T(\ell) = \ell^{10} $, this guarantees the convergence of $\sum_{\ell \geqslant 1} \mathbb{P}_{\ell}^V$ by Lemma \ref{convergence_P_v_l}. We have
      $$
      \begin{aligned}
          \mathbb{P}\big[ \mathcal{B}^{(1)}_\ell\big] & = \mathbb{P}\bigg[ \sup_{X_{\ell-1} < x_i \leqslant   X_{\ell}} \frac{| M^{(1)}_f(x_i)|}{\sqrt{x_i} R(x_i) } >  1    \bigg] 
          \\ & \leqslant \mathbb{P}\bigg[ \bigcup_{X_{\ell-1} < x_i \leqslant   X_{\ell}}\bigg\{  \frac{| M^{(1)}_f(x_i)|}{\sqrt{x_i} R(x_i) } \geqslant  1 \bigg\} \bigcap \Sigma_i \,  \bigg]+ \mathbb{P}_{\ell}^V
          \\ & \leqslant  \sum_{X_{\ell-1} < x_i \leqslant   X_{\ell}} \mathbb{P}\bigg[ \bigg\{| M^{(1)}_f(x_i)| \geqslant  \sqrt{x_i} R(x_i) \bigg\} \, \bigcap \Sigma_i \,  \bigg]+ \mathbb{P}_{\ell}^V.
      \end{aligned}
      $$
      At this stage, we can't apply Lemma \ref{hoeffding}. In fact, under the  event $\Sigma_i$, it is hard to say if $M_f^{(1)}(x_i)$ still a sum of martingale difference sequence. However the  Lemma~\ref{updated_hoeffding_chapter_1} allows us to give a strong upper bound under the event $\Sigma_i$. Thus, 
      $$
      \begin{aligned}
          \mathbb{P}\bigg[ \bigg\{ | M^{(1)}_f(x_i)| \geqslant  \sqrt{x_i} R(x_i) \bigg\} \, \bigcap \Sigma_i \,  \bigg] & \leqslant  2 \exp \bigg(\frac{-2Cx_i R(x_i)^2}{6T(\ell) \ell^{K/2}x_i} \bigg)
          \\ & \leqslant 2  \exp \bigg(-C_1 \ell^{K+2K\varepsilon-6} \bigg)
      \end{aligned}
      $$
      where $C_1$ is an absolute constant. Since by assumption $K\varepsilon = 25$, we have then 
      $$
      \sum_{X_{\ell-1} < x_i \leqslant   X_{\ell}} \exp \bigg(-C_1 \ell^{K+2K\varepsilon-6} \bigg) \leqslant  \exp \bigg( \frac{\log 2}{c_0}\ell^{K}-C_1\, \ell^{K} \ell^{44} \bigg)
      $$
    Finally, we get
      $$
      \mathbb{P}\big[\mathcal{B}_{\ell}^{(1)}\big] \ll  \exp \bigg( \frac{\log 2}{c_0}\ell^{K}-C_1\, \ell^{K} \ell^{44} \bigg) + \mathbb{P}^{V}_{\ell} 
      $$
      Thus the sum $\sum_{\ell \geqslant 1} \mathbb{P}\big[\mathcal{B}_{\ell}^{(1)}\big]$ converges.
  \end{proof}

\section*{Acknowledgement}
The author would like to thank his supervisor Régis de la Bretèche for his patient guidance, encouragement and the judicious advices he has provided throughout the work that led to this paper.
The author would also thank Gérald Tenenbaum and Adam Harper for their helpful remarks and useful
comments.

\bibliographystyle{abbrv}
\bibliography{references.bib}

\begin{center}
Université Paris Cité, Sorbonne Université
CNRS,\\
Institut de Mathématiques de Jussieu- Paris Rive Gauche,\\
F-75013 Paris, France\\
E-mail: \author{rachid.caich@imj-prg.fr}
\end{center}
\end{document}